\documentclass[10pt]{article}

\usepackage{latexsym,amsfonts,amsmath,amssymb,mathrsfs,url,color,amsthm}
\usepackage{graphicx}

\newtheorem{theorem}{Theorem}
\newtheorem{lemma}{Lemma}
\newtheorem{algorithm}{Algorithm}
\newtheorem{remark}{Remark}

\newtheorem{definition}{Definition}
\newtheorem{corollary}{Corollary}

\newcommand{\rd}{\,\mathrm{d}}
\newcommand{\rtr}{\,\mathrm{tr}}

\newcommand{\bsb}{\boldsymbol{b}}
\newcommand{\bsx}{\boldsymbol{x}}
\newcommand{\bsy}{\boldsymbol{y}}
\newcommand{\bsz}{\boldsymbol{z}}
\newcommand{\bsl}{\boldsymbol{l}}

\newcommand{\bsk}{\boldsymbol{k}}
\newcommand{\bst}{\boldsymbol{t}}
\newcommand{\bsa}{\boldsymbol{a}}
\newcommand{\bsq}{\boldsymbol{q}}
\newcommand{\bsgamma}{\boldsymbol{\gamma}}

\newcommand{\bsalpha}{\boldsymbol{\alpha}}

\newcommand{\bszero}{\boldsymbol{0}}

\newcommand{\nat}{\mathbb{N}}
\newcommand{\EE}{\mathbb{E}}
\newcommand{\RR}{\mathbb{R}}

\newcommand{\BB}{\mathcal{B}}

\newcommand{\FF}{\mathbb{F}}

\newcommand{\Dcal}{\mathcal{D}}

\newcommand{\ZZ}{\mathbb{Z}}

\newcommand{\wal}{\mathrm{wal}}
\newcommand{\Var}{\mathrm{Var}}

\allowdisplaybreaks

\begin{document}

\title{Construction of interlaced scrambled polynomial lattice rules of arbitrary high order}

\author{Takashi Goda\\ Graduate School of Engineering,\\ The University of Tokyo,\\ 7-3-1 Hongo, Bunkyo-ku, Tokyo 113-8656, Japan\\ \url{goda@frcer.t.u-tokyo.ac.jp}\\
and\\
Josef Dick\\ School of Mathematics and Statistics,\\ The University of New South Wales,\\ Sydney, NSW 2052, Australia\\ \url{josef.dick@unsw.edu.au}}

\date{\today}

\maketitle

\begin{abstract}
Higher order scrambled digital nets are randomized quasi-Monte Carlo rules which have recently been introduced in [J. Dick, Ann. Statist., 39 (2011),  1372--1398] and shown to achieve the optimal rate of convergence of the root mean square error for numerical integration of smooth functions defined on the $s$-dimensional unit cube. The key ingredient there is a digit interlacing function applied to the components of a randomly scrambled digital net whose number of components is $ds$, where the integer $d$ is the so-called interlacing factor. In this paper, we replace the randomly scrambled digital nets by randomly scrambled polynomial lattice point sets, which allows us to obtain a better dependence on the dimension while still achieving the optimal rate of convergence. Our results apply to Owen's full scrambling scheme as well as the simplifications studied by Hickernell, Matou\v{s}ek and Owen. We consider weighted function spaces with general weights, whose elements have square integrable partial mixed 
derivatives of order up to $\alpha\ge 1$, and derive an upper bound on the variance of the estimator for higher order scrambled 
polynomial lattice rules. Employing our obtained bound as a quality criterion, we prove that the component-by-component construction can be used to obtain explicit constructions of good polynomial lattice point sets. By first constructing classical polynomial lattice point sets in base $b$ and dimension $ds$, to which we then apply the interlacing scheme of order $d$, we obtain a construction cost of the algorithm of order $\mathcal{O}(dsmb^m)$ operations using $\mathcal{O}(b^m)$ memory in case of product weights, where $b^m$ is the number of points in the polynomial lattice point set.
\end{abstract}

\section{Introduction}
\label{Intro}

In this paper we study the approximation of multivariate integrals of smooth functions defined over the $s$-dimensional unit cube $[0,1]^s$,
  \begin{align*}
    I(f) =  \int_{[0,1]^s}f(\bsx)\rd \bsx ,
  \end{align*}
by averaging function values evaluated at $N$ points $\bsx_0,\ldots, \bsx_{N-1}$ with equal weights,
  \begin{align*}
    \hat{I}(f) = \frac{1}{N}\sum_{n=0}^{N-1}f(\bsx_n) .
  \end{align*}
While Monte Carlo methods choose the point set randomly, quasi-Monte Carlo (QMC) methods aim at choosing the quadrature points in a deterministic manner such that they are distributed as uniformly as possible. The Koksma-Hlawka inequality guarantees that such well-distributed point sets yield a small integration error bound, typically of order $N^{-1+\delta}$ for any $\delta > 0$, for any function which has bounded variation on $[0,1]^s$ in the sense of Hardy and Krause, see for instance \cite[Chapter~2, Section~5]{kuinie}. Digital constructions have been recognized as a powerful means of generating QMC point sets \cite{DP10,N92a}. These include the well-known constructions for digital sequences by Sobol' \cite{S67}, Faure \cite{F82}, Niederreiter \cite{N88}, Niederreiter and Xing \cite{NX} as well as others, see  \cite[Chapter~8]{DP10} for more information. Polynomial lattice point sets, first proposed in \cite{N92b}, are a special construction for digital nets and have been studied in many papers, see for 
example \cite{DKPS05,DLP05,KP07,LLNS96,LL03}. Polynomial lattice rules are QMC rules using a polynomial lattice point set as quadrature points. The major advantage of polynomial lattice rules lies in its flexibility, that is, we can design a suitable rule for the problem at hand.

In this paper we study randomized QMC rules, that is, the deterministic quadrature points are randomized such that their essential structure is retained. Owen's scrambling algorithm can be used to randomize digital nets and sequences while maintaining their equidistribution properties \cite{O95,O97a,O97b}. This not only yields a simple error estimation but also achieves a convergence of the root mean square error (RMSE) of order $N^{-3/2+\delta}$, for functions of bounded generalized variation. Since the estimator is unbiased, this can also be stated in another way, namely that the variance of the estimator decays at a rate of $N^{-3+\delta}$. It is shown in \cite{BD11} that the variance of the estimator based on a scrambled polynomial lattice rule constructed component-by-component (CBC) decays at a rate of $N^{-(2\alpha+1)+\delta}$, for functions which have bounded generalized variation of order $\alpha$ for some $0< \alpha \le 1$.

Here we consider higher smoothness, namely $\alpha \ge 1$ for which we can improve the rate of convergence of the variance of the integration error further. The initial ideas for our approach stems from the papers \cite{D07,D08,D09}. Therein higher order digital constructions of deterministic point sets and sequences were introduced whose corresponding QMC rules achieve an integration error of order $N^{-\alpha+\delta}$ for functions with square integrable partial mixed derivatives of order $\alpha\ge 1$ in each variable. An explicit construction of suitable point sets and sequences is the following interlacing algorithm. Let $d\ge 1$ and $b\ge 2$ be integers and  $\bsz\in [0,1)^{ds}$ with components $z_j=z_{j,1}b^{-1}+z_{j,2}b^{-2}+\cdots $ for $1\le j\le ds$. Then let a point $\bsx = (x_1,\ldots, x_s) \in [0,1)^s$ be given by
  \begin{align}
    x_j = \sum_{a=1}^{\infty}\sum_{r=1}^{d}z_{(j-1)d+r,a}b^{-r-(a-1)d} ,
  \label{eq:interlacing}
  \end{align}
for $1\le j\le s$. Thus, every $d$ components of $\bsz$ are interlaced to produce one component of $\bsx$. To obtain a higher order digital net or sequence, one applies the interlacing algorithm to one of the above-mentioned digital constructions. Furthermore, as shown in \cite{D11}, Owen's scrambling algorithm can be used to achieve a convergence of the variance of the estimator of order $N^{-(2\min(\alpha,d)+1)+\delta}$ for $\alpha \ge 1$. For $d\ge \alpha$, this decay rate is the best possible. For the algorithm in \cite{D11} it is important to note that one first applies Owen's scrambling to a point $\bsz$ of the digital net (or sequence) in dimension $ds$ and then interlaces the resulting point according to (\ref{eq:interlacing}) to obtain $\bsx \in [0,1]^s$. We call this method Owen's scrambling of order $d$, or order-$d$ scrambling for short here. In the proof of the convergence rate, it was assumed in \cite{D11} that the underlying point set is explicitly given by some digital $(t,m,ds)$-net or $(t,
ds)$-sequence. The $t$-value of digital $(t,ds)$-sequences, however, grows at least linearly with $s$, and consequently, it becomes hard to obtain a bound of the variance independent of the dimension.

In this paper, we study order-$d$ scrambled polynomial lattice point sets for numerical integration. Our strategy is to construct classical polynomial lattice rules in dimension $ds$ using a suitable quality criterion, then apply Owen's scrambling to the quadrature points of the polynomial lattice rule and finally to apply the interlacing algorithm of order $d$ to obtain a randomized quadrature rule for the domain $[0,1]^s$. We refer to such quadrature rules by {\it interlaced scrambled polynomial lattice rules}. The major contributions of our study are to derive a computable upper bound on the variance of the estimator for higher order scrambled digital nets, which is an extension of the study in \cite{D11}, and by employing our obtained bound as a quality criterion, to prove that the CBC construction can be used to obtain good polynomial lattice rules. Through our argument we need to overcome several non-trivial technical difficulties specific to the interlacing algorithm. The resulting advantage compared 
to the results in \cite{D11} is the weaker dependence on the dimension and the possibility to construct the rules for a given set of weights when the integrand has finite weighted bounded variation, see Subsection~\ref{ssec:var}. As in \cite{SW98}, the weights model the dependence of the integrand on certain projections. With our approach, we are able to obtain tractability results under certain conditions on the weights. Furthermore, our results also apply to the simplified scrambling schemes studied by Hickernell~\cite{H96}, Matou\v{s}ek~\cite{Mat98} and Owen~\cite{O03}. Thus efficient implementations of the scrambling procedure are available for our interlaced scrambled polynomial lattice rules.

As in \cite{D11}, the upper bound on the variance in this paper is, apart from the factor $N^\delta$, optimal in terms of the dependence on the number of points (see \cite{No88}), and compared to \cite{D11} improves the dependence of the upper bound on the dimension. We are not aware of any other randomized equal weight quadrature rule with the properties shown in this paper. An alternative (in general, non-equal weight) algorithm based on Monte Carlo and `separation of the main part' is for instance discussed in \cite[Section~7.4]{MNR}. This algorithm also achieves the optimal rate of convergence in terms of the number of points, but they do not discuss the dependence of this method on the dimension. In fact, \cite[Open Problem~91]{NW10} asks for the precise condition on the weights such that one obtains an upper bound independent of the dimension for a certain Sobolev space of smoothness $\alpha = 1$. Corollary~\ref{cor1} below provides an upper bound which is independent of the dimension for a different function space, however, we do not know whether our result is also best possible.

In the next section we describe the necessary background and notation, namely polynomial lattice rules, Owen's scrambling, and higher order digital constructions. We also describe the main results of the paper. Namely we introduce a component-by-component algorithm, state a result on the convergence behavior of the interlaced scrambled polynomial lattice rule and discuss randomized QMC tractability. In Section~\ref{Variance} we derive an upper bound on the variance of the estimator in the weighted function space with general weights where a function has square integrable partial mixed derivatives of order $\alpha\ge 1$ in each variable. Using this bound we show how the quality criterion for the construction of interlaced scrambled polynomial lattice rules is derived. In Section~\ref{CBC construction} we prove that interlaced scrambled polynomial lattice rules constructed using the CBC algorithm can achieve a convergence of the variance of the estimator of order $N^{-(2\min(\alpha,d)+1)+\delta}$. Thereafter we assume 
product weights for simplicity of exposition and describe the fast CBC construction by using the fast Fourier transform as introduced in \cite{NC06a,NC06b}. We show that the interlaced 
scrambled polynomial lattice rules in
base $b$ can be constructed in order $\mathcal{O}(dsmb^m)$ operations using order $\mathcal{O}(b^m)$ memory, where $b^m$ is the number of points in $[0,1]^s$. This is a significant reduction in the construction cost to previously obtained component-by-component algorithms for higher order polynomial lattice rules~\cite{BDLNP}. We conclude this paper with numerical experiments in Section~\ref{Experiments}.

\section{Background, notation and results}
\label{Background}

In this section, as necessary tools for our study, we  introduce polynomial lattice rules, Owen's scrambling algorithm, and higher order digital net constructions. Thereafter, we describe the main results of the paper.

Let $\nat$ denote the set of positive integers and $\nat_0$ denote the set of non-negative integers. For $i,j\in \nat$ such that $i\le j$, we denote by $\{i:j\}$ the index set $\{i,i+1,\ldots, j-1,j\}$. For a prime $b$, let $\FF_b$ be the finite field containing $b$ elements $\{0,\ldots, b-1 \}$. For simplicity we identify the elements of $\FF_b$ with the integers $0, 1, \ldots, b-1 \in \ZZ$.

\subsection{Polynomial lattice rules}

We introduce some notation first. For a prime $b$, we denote by $\FF_b((x^{-1}))$ the field of formal Laurent series over $\FF_b$. Every element of $\FF_b((x^{-1}))$ is of the form
  \begin{align*}
    L = \sum_{l=w}^{\infty}t_l x^{-l} ,
  \end{align*}
where $w$ is some integer and all $t_l\in \FF_b$. Further, we denote by $\FF_b[x]$ the set of all polynomials over $\FF_b$. For a given integer $m$, we define the map $v_m$ from $\FF_b((x^{-1}))$ to the interval $[0,1)$ by
  \begin{align*}
    v_m\left( \sum_{l=w}^{\infty}t_l x^{-l}\right) =\sum_{l=\max(1,w)}^{m}t_l b^{-l}.
  \end{align*}
We often identify $k\in \nat_0$, whose $b$-adic expansion is given by $k=\kappa_0+\kappa_1 b+\cdots +\kappa_{a-1} b^{a-1}$, with the polynomial over $\FF_b[x]$ given by $k(x)=\kappa_0+\kappa_1 x+\cdots +\kappa_{a-1} x^{a-1}$.  For $\bsk=(k_1,\ldots, k_s)\in (\FF_b[x])^s$ and $\bsq=(q_1,\ldots, q_s)\in (\FF_b[x])^s$, we define the 'inner product' as
  \begin{align}\label{eq:inner_product}
     \bsk \cdot \bsq =\sum_{j=1}^{s}k_j q_j \in \FF_b[x] ,
  \end{align}
and we write $q\equiv 0 \pmod p$ if $p$ divides $q$ in $\FF_b[x]$.

The definition of a polynomial lattice rule is given as follows.
\begin{definition}\label{def:polynomial_lattice}
Let $b$ be a prime and $m, s \in \nat$. Let $p \in \FF_b[x]$ such that $\deg(p)=m$ and let $\bsq=(q_1,\ldots,q_s) \in (\FF_b[x])^s$. Now we construct a point set consisting of $b^m$ points in $[0,1)^s$ in the following way:  For $0 \le n < b^m$, identify each $n$ with a polynomial $n(x)\in \FF_b[x]$ of $\deg(n(x))<m$. Then the $n$-th point is obtained by setting
  \begin{align*}
    \bsx_n
    &:=
    \left( v_m\left( \frac{n(x) \, q_1(x)}{p(x)} \right) , \ldots , v_m\left( \frac{n(x) \, q_s(x)}{p(x)} \right) \right) \in [0,1)^s .
  \end{align*}
The point set $\{\bsx_0, \ldots, \bsx_{b^m-1}\}$ is called a \emph{polynomial lattice point set} and a QMC rule using this point set is called a \emph{polynomial lattice rule} with generating vector $\bsq$ and modulus $p$.
\end{definition}

We add one more notation and introduce the concept of the so-called {\it dual polynomial lattice} of a polynomial lattice point set. For $k\in \nat_0$ with $b$-adic expansion $k= \kappa_0 + \kappa_1 b+\cdots + \kappa_{a-1} b^{a-1}$, let $\rtr_m(k)$ be the polynomial of degree at most $m$ obtained by truncating the associated polynomial $k(x)\in \FF_b[x]$ as
  \begin{align*}
    \rtr_m(k)= \kappa_0 + \kappa_1 x+\cdots + \kappa_{m-1}x^{m-1},
  \end{align*}
where we set $\kappa_{a} = \cdots = \kappa_{m-1} = 0$ if $a< m$.
For a vector $\bsk=(k_1,\ldots, k_s)\in \nat_0^s$, we define $\rtr_m(\bsk)=(\rtr_m(k_1),\ldots, \rtr_m(k_s))$. With this notation, we introduce the following definition of the dual polynomial lattice $D^{\perp}$.
\begin{definition}\label{def:dual_net}
The dual polynomial lattice of a polynomial lattice point set with modulus $p\in \FF_b[x]$, $\deg(p)=m$, and generating vector $\bsq \in (\FF_b[x])^s$ is given by
  \begin{align*}
     D^\perp  = \{ \bsk\in \nat_0^{s}:\ \mathrm{tr}_m(\bsk)\cdot \bsq\equiv 0 \pmod p \} ,
  \end{align*}
where the inner product is in the sense of (\ref{eq:inner_product}).
\end{definition}

\subsection{Owen's scrambling}

We now introduce Owen's scrambling algorithm. This procedure is best explained by using only one point $\bsx$. We denote the point obtained after scrambling $\bsx$ by $\bsy$. For $\bsx=(x_1,\ldots, x_s)\in [0,1)^s$, we denote the $b$-adic expansion by
  \begin{align*}
     x_j=\frac{x_{j,1}}{b}+\frac{x_{j,2}}{b^2}+\cdots ,
  \end{align*}
for $1\le j\le s$, where we assume that for each $1 \le j \le s$ infinitely many digits $x_{j,k}$ are different from $b-1$. Let $\bsy=(y_1,\ldots, y_s)\in [0,1)^s$ be the scrambled point whose $b$-adic expansion is represented by
  \begin{align*}
     y_j=\frac{y_{j,1}}{b}+\frac{y_{j,2}}{b^2}+\cdots ,
  \end{align*}
for $1\le j\le s$. Each coordinate $y_j$ is obtained by applying permutations to each digit of $x_j$. Here the permutation applied to $x_{j,k}$ depends on $x_{j,l}$ for $1\le l\le k-1$. In particular, $y_{j,1}=\pi_j(x_{j,1})$, $y_{j,2}=\pi_{j,x_{j,1}}(x_{j,2})$, $y_{j,3}=\pi_{j,x_{j,1},x_{j,2}}(x_{j,3})$, and in general
  \begin{align*}
     y_{j,k}=\pi_{j,x_{j,1},\ldots, x_{j,k-1}}(x_{j,k}) ,
  \end{align*}
where $\pi_{j,x_{j,1},\ldots, x_{j,k-1}}$ is a random permutation of $\{0,\ldots, b-1\}$. We choose permutations with different indices mutually independent from each other where each permutation is chosen uniformly distributed. Then, as shown in \cite[Proposition~2]{O95}, the scrambled point $\bsy$ is uniformly distributed in $[0,1)^s$.

In order to simplify the notation, we denote by $\Pi_j$ the set of permutations associated with $x_j$, that is,
  \begin{align*}
     \Pi_j=\{ \pi_{j,x_{j,1},\ldots, x_{j,k-1}}: k\in \nat, x_{j,1},\ldots, x_{j,k-1}\in \{0,\ldots, b-1 \}\} ,
  \end{align*}
and let $\boldsymbol{\Pi}$ be a set of $\Pi_1 \cup \ldots \cup \Pi_s$. With an abuse of notation we simply write $\bsy=\boldsymbol{\Pi}(\bsx)$ when $\bsy$ is obtained by applying Owen's scrambling to $\bsx$ using the permutations in $\boldsymbol{\Pi}$.

As can be seen from the above description, Owen's original scrambling is quite expensive to compute. In order to reduce the computational cost, various simplified scrambling schemes have been introduced which can be implemented more easily, see for example \cite{H96,Mat98,O03}. Although we only deal with Owen's original scrambling in the remainder of this paper, the simplified scramblings cited above also apply here as long as they satisfy so-called Owen's lemma \cite[Lemma~6]{D11}. 

\subsection{Higher order digital nets}

Quasi-Monte Carlo rules based on higher order digital nets exploit the smoothness of an integrand so that they can achieve the optimal order of convergence of the integration error for functions with smoothness $\alpha\in \nat$. The result is based on a bound on the decay of the Walsh coefficients of smooth functions \cite{D08}. We refer readers to \cite{D09} for a brief introduction of the central ideas. Explicit constructions of higher order digital nets and sequences were given in \cite{D08}.

There is also a component-by-component construction algorithm of higher order polynomial lattice rules \cite{BDGP11}. Higher order polynomial lattice rules can be obtained in the following way. In Definition \ref{def:polynomial_lattice}, we set $p$ with $\deg(p)=n>m$ and replace $v_m$ with $v_n$ for the mapping function. Then a higher order polynomial lattice point set consists of the first $b^m$ points of a classical polynomial lattice point set with $b^n$ points (where $n = \alpha m$ for integrands of smoothness $\alpha$). The existence of higher order polynomial lattice rules achieving the optimal order of convergence was established in \cite{DP07} and the CBC construction was proved to achieve the optimal order of convergence in \cite{BDGP11}. However, we have no generalization of the scrambling algorithm to higher order polynomial lattice rules which preserves the higher order structure. To work around this problem we use a different approach in this paper. Namely, we use the approach from \cite{D07,
D08} based on the interlacing of digital nets or sequences.

We describe the interlacing algorithm in more detail in the following. Since the interlacing is applied to each point separately, we use just one point to describe the procedure. Let $\bsz \in [0,1)^{ds}$, with  $\bsz = (z_{1},\ldots, z_{ds})$ and consider the $b$-adic expansion of each coordinate
  \begin{align*}
    z_{j}= \frac{z_{j,1}}{b} + \frac{z_{j,2}}{b^2} + \cdots ,
  \end{align*}
unique in the sense that infinitely many digits are different from $b-1$. We obtain a point $\bsx \in [0,1)^{s}$ by interlacing the digits of $d$ components of $\bsz$ in the following way: Let $\bsx =(x_{1},\ldots, x_{s})$, where
  \begin{align*}
    x_{j} = \sum_{a=1}^{\infty}\sum_{r=1}^{d} z_{(j-1)d+r,a}b^{-r-(a-1)d},
  \end{align*}
for $1\le j\le s$. We denote this mapping by $\Dcal_{d}: [0,1)^d \to [0,1)$ and we simply write $x_{j}=\Dcal_{d}(z_{(j-1)d+1},\ldots, z_{jd})$. Further we write
  \begin{align*}
    \bsx=\Dcal_d(\bsz) := (\Dcal_d(z_{1},\ldots, z_d), \Dcal_d(z_{d+1},\ldots, z_{2d}), \ldots, \Dcal_d(z_{(s-1)d+1},\ldots, z_{sd})) ,
  \end{align*}
when $\bsx$ is obtained by interlacing the components of $\bsz$. Note that the interlacing procedure depends on the base $b$. Throughout the paper we assume that the construction of polynomial lattice rules, Owen's scrambling and the interlacing of digits all use the same prime base $b$.

Order-$d$ scrambling for $\bsz \in [0,1)^{ds}$ proceeds as follows. Let $\boldsymbol{\Pi}$ be a uniformly chosen i.i.d. set of permutations. Then, the order-$d$ scrambled point $\bsy \in [0,1)^s$ is given by
  \begin{align*}
   \bsy = \Dcal_d(\boldsymbol{\Pi}(\bsz)).
\end{align*}
Hence, as stated in the previous section, we first apply Owen's scrambling to $\bsz \in [0,1)^{ds}$ and then interlace the digits of the resulting point to obtain the point $\bsy$. Again, we choose permutations with different indices mutually independent from each other where each permutation is chosen with the same probability. Then, as shown in \cite[Proposition~5]{D11}, the order-$d$ scrambled point $\bsy$ is uniformly distributed in $[0,1)^s$.

In this paper, we are interested in the use of polynomial lattice rules to generate a point set in $[0,1)^{ds}$. For clarity, we give the definition of interlaced scrambled polynomial lattice rules below.

\begin{definition}\label{def:interlacing_polynomial_lattice}
Let $b$ be a prime and $m, s, d \in \nat$. Let $p \in \FF_b[x]$ such that $\deg(p)=m$ and let $\bsq=(q_1,\ldots,q_{ds}) \in (\FF_b[x])^{ds}$. Now we construct a point set consisting of $b^m$ points in $[0,1)^s$.  For $0 \le n < b^m$, the $n$-th point is obtained by setting
  \begin{align*}
    \bsz_n = \left( v_m\left( \frac{n(x) \, q_1(x)}{p(x)} \right) , \ldots , v_m\left( \frac{n(x) \, q_{ds}(x)}{p(x)} \right) \right) \in [0,1)^{ds} .
  \end{align*}
Then let
  \begin{align*}
    \bsy_n = \Dcal_d(\boldsymbol{\Pi}(\bsz_n)),
  \end{align*}
where the permutations are chosen independently and uniformly distributed from the set $\boldsymbol{\Pi}$. We call $\{\bsy_0, \ldots, \bsy_{b^m-1}\}$ an \emph{interlaced scrambled polynomial lattice point set} (of order $d$) and a QMC rule using the point set $\{\bsy_0, \ldots, \bsy_{b^m-1}\}$ an \emph{interlaced scrambled polynomial lattice rule} (of order $d$).
\end{definition}

\subsection{The results}

We now describe the main results of this paper. In the following, let $\{\bsy_0, \ldots, \bsy_{b^m-1}\} \subset [0,1]^s$ be an interlaced scrambled polynomial lattice point set. Further, we denote by $\{\bsz_0, \ldots, \bsz_{b^m-1}\} \subset [0,1]^{ds}$ the polynomial lattice point set with $ds$ components with modulus $p \in \FF_b[x]$, $\deg(p)=m$, and with generating vector $\bsq=(q_1,\ldots,q_{ds}) \in (\FF_b[x])^{ds}$. We approximate the integral $I(f) =  \int_{[0,1]^s}f(\bsx)\rd\bsx$ by
  \begin{align*}
    \hat{I}(f) =  \frac{1}{b^m}\sum_{n=0}^{b^m-1}f(\bsy_n).
  \end{align*}
Since the order-$d$ scrambled point is uniformly distributed in $[0,1)^s$ as shown in \cite[Proposition~5]{D11}, this estimator is unbiased, that is, $\EE[\hat{I}(f)]=I(f)$. It follows that the mean square error equals the variance of the estimator. Thus, in the following, we concentrate on the variance of the estimator denoted by
  \begin{align*}
    \Var[\hat{I}(f)] = \EE\left[ \left( \hat{I}(f)-\EE[\hat{I}(f)] \right)^2 \right].
  \end{align*}

Let $\bsgamma = (\gamma_u)_{u \subseteq \{1:s\}}$ be a vector of nonnegative real numbers. These numbers are called weights and are used to model the importance of projections of functions $f:[0,1]^s \to \mathbb{R}$, where $\gamma_u$ small means that the projection of $f$ onto the components in $u$ is of little importance and vice versa. This idea stems from  \cite{SW98}. For more details see Subsection~\ref{ssec:var} below.  We assume that $f$ has smoothness $\alpha \in \mathbb{N}$ which we make precise by the assumption that $V_{\alpha, \bsgamma}(f) < \infty$. Here $V_{\alpha,\bsgamma}(f)$ is a variation of order $\alpha$ which can be related to a Sobolev norm where partial derivatives of order up to $\alpha$ in each variable are square integrable. See Subsection~\ref{ssec:var} for details on $V_{\alpha,\bsgamma}(f)$. In Corollary~\ref{corollary:bound_variance} we show that for an interlaced scrambled polynomial lattice rule of order $d$ the variance of the estimator is bounded by
\begin{equation*}
\Var[\hat{I}(f)] \le V_{\alpha, \bsgamma}^2(f) B_{\alpha,d, \bsgamma}(\bsq, p),
\end{equation*}
where $B_{\alpha,d, \bsgamma}(\bsq, p)$ is a function which depends only on the interlaced scrambled polynomial lattice rule but does not depend on $f$. The precise formula for $B_{\alpha, d, \bsgamma}(\bsq, p)$ is derived in Subsection~\ref{sec_bound_var}. In Lemma~\ref{lemma:criterion} we show that there is a concise formula for $B_{\alpha, d, \bsgamma}(\bsq, p)$ given by
  \begin{align*}
     B_{\alpha,d, \bsgamma}(\bsq, p) = \frac{1}{b^m}\sum_{n=0}^{b^m-1}\sum_{\emptyset\ne v\subseteq \{1:s\}} \gamma_{v}D_{\alpha,d}^{|v|} \prod_{j\in v}\Big[ -1+\prod_{k=1}^{d}\left(1+\phi_{\alpha,d}(z_{n,(j-1)d+k})\right)\Big],
  \end{align*}
where $D_{\alpha,d}:=4^{\max(d-\alpha,0)}b^{(2d-1)\alpha}$, and for $z\in [0,1)$, let
  \begin{align*}
     \phi_{\alpha,d}(z):= \frac{(b-1)\left(b-1-b^{2\min(\alpha,d)\lfloor \log_b z\rfloor}(b^{2\min(\alpha,d)+1}-1)\right)}{b^{\alpha}(b^{2\min(\alpha,d)}-1)} ,
  \end{align*}
where we set $b^{2\min(\alpha,d)\lfloor \log_b 0\rfloor}=0$. In particular, for product weights $\gamma_v=\prod_{j\in v}\gamma_j$, we have
  \begin{align*}
     B_{\alpha,d, \bsgamma}(\bsq, p) = -1+\frac{1}{b^m}\sum_{n=0}^{b^m-1}\prod_{j=1}^{s} \Big[ 1- \gamma_j D_{\alpha,d} + \gamma_j D_{\alpha,d} \prod_{k=1}^{d}\left(1+\phi_{\alpha,d}(z_{n,(j-1)d+k})\right)\Big] .
  \end{align*}
Thus $B_{\alpha, d, \bsgamma}(\bsq, p)$ can be used as a quality criterion for searching for good generating vectors. In the following we introduce the CBC algorithm.

The CBC construction algorithm was first introduced by Korobov \cite{K59}, and independently reinvented later by Sloan and Reztsov \cite{SR02}, to construct a generating vector of lattice rules. The same approach can be applied to polynomial lattice rules. In the following, we choose an irreducible polynomial $p$ such that $\deg(p)=m$, and restrict $q_j$, $1 \le j \le ds$, to non-zero polynomials over $\FF_b$ such that its degree is less than $m$. Without loss of generality we can set $q_1=1$. We denote by $R_{b,m}$ the set of all non-zero polynomials over $\FF_b$ with degree less than $m$, i.e.,
	\begin{align*}
		R_{b,m}=\{ q\in \FF_b[x]: \deg(q)<m\ \mathrm{and}\ q\ne0\} .
	\end{align*}
We note that $|R_{b,m}|=b^m-1$. Further, we write $\bsq_{\tau}=(q_1,\ldots, q_{\tau})$ for $1\le \tau\le ds$. The idea is now to search for the polynomials $q_j \in R_{b,m}$ component-by-component. To do so, we need to define $B_{\alpha, d, \bsgamma}(\bsq_\tau, p)$ for arbitrary $1 \le \tau \le ds$. This is done in the following way. Let $1 \le \tau \le ds$ and $\beta = \lceil \tau/d \rceil$. Then
  \begin{align*}
    & B_{\alpha,d, \bsgamma}(\bsq_\tau, p) \\ = & \frac{1}{b^m}\sum_{n=0}^{b^m-1}\sum_{\emptyset\ne v\subseteq \{1:\beta-1\}} \gamma_{v}D_{\alpha,d}^{|v|}\prod_{j\in v}\Big[ -1+\prod_{k=1}^{d}\left(1+\phi_{\alpha,d}(z_{n,(j-1)d+k})\right)\Big] \\ & + \frac{1}{b^m}\sum_{n=0}^{b^m-1}\sum_{\{\beta\} \subseteq v \subseteq \{1:\beta\}} \gamma_{v}D_{\alpha,d}^{|v|}\prod_{j\in v \setminus \{\beta\} }\Big[ -1+\prod_{k=1}^{d}\left(1+\phi_{\alpha,d}(z_{n,(j-1)d+k})\right)\Big] \\ & \times \Big[ -1+\prod_{k=1}^{\tau - (\beta-1)d}\left(1+\phi_{\alpha,d}(z_{n,(\beta-1)d+k})\right)\Big] .
  \end{align*}
For product weights we have
  \begin{align*}
&  B_{\alpha,d, \bsgamma}(\bsq_\tau, p) \\ = &  -1+\frac{1}{b^m}\sum_{n=0}^{b^m-1}\prod_{j=1}^{\beta-1} \Big[ 1- \gamma_j D_{\alpha,d} + \gamma_j D_{\alpha,d} \prod_{k=1}^{d}\left(1+\phi_{\alpha,d}(z_{n,(j-1)d+k})\right)\Big] \\ & \times \Big[ 1- \gamma_\beta D_{\alpha,d} +  \gamma_\beta D_{\alpha,d} \prod_{k=1}^{\tau-(\beta-1) d}\left(1+\phi_{\alpha,d}(z_{n,(\beta-1)d+k})\right)\Big] .
  \end{align*}

The CBC construction intended for this study proceeds as follows.
\begin{algorithm}\label{algorithm:cbc}
For a prime base $b$, a dimension $s$, an interlacing factor $d$, and integer $m\ge 1$ and weights $\bsgamma=(\gamma_u)_{u\subseteq \{1:s\}}$:
	\begin{enumerate}
		\item Choose an irreducible polynomial $p\in \FF_b[x]$ with $\deg(p)=m$.
		\item Set $q_1=1$.
		\item For $\tau=2,\ldots, ds$, find $q_{\tau}$ by minimizing $B_{\alpha,d, \bsgamma}((\bsq_{\tau-1}, \tilde{q}_{\tau}),p)$ as a function of $\tilde{q}_{\tau}\in R_{b,m}$.
	\end{enumerate}
\end{algorithm}
In Subsection~\ref{subsec_fast_cbc} we show that one can also use the fast CBC algorithm of \cite{NC06a,NC06b} to find good generating vectors.
 
Next we show that the generating vector found by Algorithm~\ref{algorithm:cbc} satisfies the bound in the following theorem.

\begin{theorem}\label{theorem:cbc_proof}
Let $b$ be a prime and $p\in \FF_b[x]$ be irreducible with $\deg(p)=m$. Suppose that $\bsq=(q_1,\ldots, q_{ds})$ is constructed using Algorithm \ref{algorithm:cbc}. Then, for all $\tau=1,\ldots, ds$ we have
\begin{align*}
& B_{\alpha,d, \bsgamma}(\bsq_{\tau},p) \\ \le  
& \frac{1}{(b^m-1)^{1/\lambda}} \left[ \sum_{\emptyset \ne u\subseteq \{1:j_0-1\}}\gamma_u^\lambda C_{\alpha,d,\lambda,d}^{|u|}+C_{\alpha,d,\lambda,d_0}\sum_{u\subseteq \{1:j_0-1\}}\gamma_{u\cup\{j_0\}}^\lambda C_{\alpha,d,\lambda,d}^{|u|}\right]^{1/\lambda} ,
\end{align*}
for all $1/(2\min(\alpha,d)+1)<\lambda\le 1$, where $\tau=(j_0-1)d+d_0$ such that $j_0, d_0 \in \nat$ and $0< d_0\le d$,
\begin{align*}
 C_{\alpha,d,\lambda,a}= D_{\alpha,d}^{\lambda}\left(-1+(1+\tilde{C}_{\alpha,d,\lambda})^a\right) ,
\end{align*}
and
\begin{align*}
 \tilde{C}_{\alpha,d,\lambda}=\max\left\{\left(\frac{(b-1)^2}{b^{\alpha}(b^{2\min(\alpha,d)}-1)}\right)^{\lambda}, \frac{(b-1)^{1+\lambda}}{b^{\lambda(\alpha-1)}(b^{(2\min(\alpha,d)+1)\lambda}-b)}\right\} .
\end{align*}	
\end{theorem}
The proof of this result is presented in Subsection~\ref{subsec:proofthm1}. 

By choosing the interlacing factor $d \ge \alpha$, Theorem~\ref{theorem:cbc_proof} implies a convergence rate of the variance $\Var[\hat{I}(f)]$ of order $N^{-2\alpha - 1 + \delta}$, for any $\delta > 0$. This rate of convergence is essentially best possible as explained in \cite{DG13} (which follows by relating $V_{\alpha, \bsgamma}$ to a Sobolev norm and then using \cite[Section~2.2.9, Proposition~1(ii)]{No88}.)

We discuss now the randomized QMC tractability properties of our constructed interlaced scrambled polynomial lattice rules. In the concept of tractability of multivariate problems, we study the dependence of $B_{\alpha,d, \bsgamma}(\bsq, p)$ on the dimension $s$ and the number of points $N=b^m$. Especially we are interested in the case when $B_{\alpha,d, \bsgamma}(\bsq, p)$ does not depend on $s$ and the case when $B_{\alpha,d, \bsgamma}(\bsq, p)$ depends polynomially on $s$. We restrict ourselves to randomized QMC rules $Q_{N,s}$ and set
\begin{equation*}
N(\varepsilon, s) = \min \{N \in \mathbb{N}: \sqrt{\EE\left[ (I(f) - Q_{N,s}(f))^2\right]} < \varepsilon \mbox{ for all } f \mbox{ with } V_{\alpha, \bsgamma}(f) \le 1\}.
\end{equation*}
Then randomized QMC polynomial tractability means that for all $\varepsilon > 0$ and $s \in \mathbb{N}$ we have
\begin{equation*}
N(\varepsilon, s) \le C \varepsilon^{-p} s^q
\end{equation*}
for some $p, q > 0$ and randomized QMC strong polynomial tractability means that the above bound holds for $q=0$. In the following we show randomized QMC strong polynomial tractability and randomized QMC polynomial tractability under certain conditions on the weights by showing that the bound $B_{\alpha, d, \bsgamma}$ is bounded independently of the dimension or depends at most polynomially on the dimension. A comprehensive introduction to tractability studies can be found in \cite{NW08, NW10}.

For $p$ and $\bsq$ constructed according to Algorithm \ref{algorithm:cbc}, we have from Theorem~\ref{theorem:cbc_proof} for $\tau = ds$
	\begin{align*}
	B_{\alpha,d, \bsgamma}(\bsq,p) \le \frac{1}{(N-1)^{1/\lambda}} \left[ \sum_{\emptyset \ne u\subseteq \{1:s\}}\gamma_u^\lambda C_{\alpha,d,\lambda,d}^{|u|}\right]^{1/\lambda} ,
	\end{align*}
for all $1/(2\min(\alpha,d)+1)<\lambda\le 1$. In case of product weights $\gamma_u=\prod_{j\in u}\gamma_j$, we have
	\begin{align*}
	B_{\alpha,d, \bsgamma}(\bsq,p) & \le \frac{1}{(N-1)^{1/\lambda}} \left[ \sum_{\emptyset \ne u\subseteq \{1:s\}}\prod_{j\in u}\gamma_j^\lambda C_{\alpha,d,\lambda,d}\right]^{1/\lambda} \\
    & = \frac{1}{(N-1)^{1/\lambda}} \left[-1+\prod_{j=1}^{s}\left( 1+\gamma_j^\lambda C_{\alpha,d,\lambda,d}\right)\right]^{1/\lambda} .
	\end{align*}
Since the term in the bracket of these bounds is independent of the number of points, $B_{\alpha,d, \bsgamma}(\bsq,p)$ depends polynomially on the number of points with its degree $-(2\min(\alpha,d)+1)<-1/\lambda\le -1$. We then have the following corollary of Theorem \ref{theorem:cbc_proof}.
\begin{corollary}\label{cor1}
Let $b$ be a prime base, $p\in \FF_b[x]$ be irreducible with $\deg(p)=m$. Suppose that $\bsq$ is constructed according to Algorithm \ref{algorithm:cbc}. Then we have the following:
\begin{enumerate}
\item For general weights, assume that
		\begin{align*}
		\lim_{s\to \infty}\sum_{\emptyset \ne u\subseteq \{1:s\}}\gamma_{u}^{\lambda}C_{\alpha,d,\lambda,d}^{|u|} < \infty ,
		\end{align*}
for some $1/(2\min(\alpha,d)+1)<\lambda\le 1$. Then $B_{\alpha,d,\bsgamma}(\bsq,p)$ is bounded independently of the dimension.

\item For general weights, assume that
		\begin{align*}
		\limsup_{s\to \infty}\left[\frac{1}{s^q}\sum_{\emptyset \ne u\subseteq \{1:s\}}\gamma_{u}^{\lambda}C_{\alpha,d,\lambda,d}^{|u|}\right] < \infty ,
		\end{align*}
for some $1/(2\min(\alpha,d)+1)<\lambda\le 1$ and $q>0$. Then the bound of $B_{\alpha,d,\bsgamma}(\bsq,p)$ depends polynomially on the dimension with its degree $q/\lambda$.

\item For product weights $\gamma_u=\prod_{j\in u}\gamma_j$, assume that
		\begin{align*}
		\sum_{j=1}^{\infty}\gamma_{j}^{\lambda} < \infty ,
		\end{align*}
for some $1/(2\min(\alpha,d)+1)<\lambda\le 1$. Then $B_{\alpha,d,\bsgamma}(\bsq,p)$ is bounded independently of the dimension.

\item For product weights $\gamma_u=\prod_{j\in u}\gamma_j$, assume that
		\begin{align*}
		A:= \limsup_{s\to \infty}\frac{\sum_{j=1}^{s}\gamma_j}{\log s} < \infty .
		\end{align*}
Then the bound of $B_{\alpha,d,\bsgamma}(\bsq,p)$ depends polynomially on the dimension with its degree $C_{\alpha,d,1,d}(A+\eta)$ for any $\eta>0$.

\end{enumerate}
\end{corollary}

\begin{proof}
It is straightforward to have the results for general weights as in the proof of \cite[Theorem~3]{DSWW06} and the results for product weights by following the similar lines as the proof of \cite[Corollary~4.5]{DKPS05}.
\end{proof}

Further implications for tractability in the infinite dimensional setting of the results in this paper are discussed in more detail in \cite{DG13}, where in some cases optimal tractability results for so-called changing dimension algorithms were obtained.

\section{Variance of the estimator}
\label{Variance}

To analyze the variance of the estimator we use Walsh functions, which we introduce in the next subsection.

\subsection{Walsh functions}

Walsh functions were first introduced in \cite{W23} for the case of base 2 and were generalized later, see for instance \cite{C55}. We first give the definition for the one-dimensional case.
\begin{definition}
Let $b\ge 2$ be an integer and $\omega_b=e^{2\pi \mathrm{i}/b}$. We represent $k\in \nat_0$ in base $b$, $k = \kappa_0+\kappa_1 b+\cdots +\kappa_{a-1}b^{a-1}$ with $\kappa_z\in \{0, 1, \ldots, b-1\}$. Then, the $k$-th $b$-adic Walsh function ${}_b\wal_k: [0,1)\to \{1,\omega_b,\ldots, \omega_b^{b-1}\}$ is defined as
  \begin{align*}
    {}_b\wal_k(x)
    =
    \omega_b^{x_1\kappa_0+\cdots+x_a\kappa_{a-1}} ,
  \end{align*}
for $x\in [0,1)$ with $b$-adic expansion $x=x_1 b^{-1} + x_2 b^{-2} + \cdots$, unique in the sense that infinitely many of the $x_z$ are different from $b-1$.
\end{definition}
This definition can be generalized to higher dimensions.

\begin{definition}
For dimension $s\ge 2$, let $\bsx=(x_1,\ldots, x_s)\in [0,1)^s$ and $\bsk=(k_1,\ldots, k_s)\in \nat_0^s$. We define ${}_b\wal_{\bsk}: [0,1)^s \to \{1,\omega_b,\ldots, \omega_b^{b-1}\}$ by
  \begin{align*}
    {}_b\wal_{\bsk}(\bsx)
    =
    \prod_{j=1}^s {}_b\wal_{k_j}(x_j) .
  \end{align*}
\end{definition}
Since we will always use Walsh functions in a fixed base $b$ in the rest of this paper, we omit the subscript and simply write $\wal_k$ or $\wal_{\bsk}$.

The following important lemma relates the dual polynomial lattice to numerical integration of Walsh functions.
\begin{lemma}\label{lamma:dual_walsh}
Let $\{\bsx_0,\ldots, \bsx_{b^m-1} \}$ be a polynomial lattice point set with modulus $p\in \FF_b[x]$, $\deg(p)=m$, and generating vector $\bsq \in (\FF_b[x])^s$ and let $D^{\perp}$ be its dual polynomial lattice. Then we have
  \begin{align*}
    \frac{1}{b^m}\sum_{n=0}^{b^m-1}\wal_{\bsk}(\bsx_n)=\left\{ \begin{array}{ll}
     1 & \mathrm{if}\ \bsk\in D^{\perp } , \\
     0 & \mathrm{otherwise} . \\
     \end{array} \right.
  \end{align*}
\end{lemma}

\begin{proof}
This follows immediately from Definition \ref{def:dual_net}, \cite[Lemma~10.6]{DP10} and \cite[Lemma~4.75]{DP10}.
\end{proof}

\subsection{Variance estimates}

We consider the following Walsh series expansion for $f\in L_2([0,1]^s)$
  \begin{align*}
    f(\bsx) \sim \sum_{\bsk\in \nat_0^s}\hat{f}(\bsk)\mathrm{wal}_{\bsk}(\bsx) ,
  \end{align*}
where the Walsh coefficients $\hat{f}(\bsk)$ are given by
  \begin{align*}
    \hat{f}(\bsk) = \int_{[0,1]^s}f(\bsx)\overline{\mathrm{wal}_{\bsk}(\boldsymbol{x})}\rd \bsx .
  \end{align*}

The following notation is needed for deriving the lemma below. Let $\bsl=(\bsl_1,\ldots, \bsl_s) \in \nat_0^{ds}$ where $\bsl_j=(l_{(j-1)d+1},\ldots, l_{jd})$ and
  \begin{align}\label{eq:BB_d}
     \BB_{d,\bsl,s}=\{ (k_1,\ldots, k_{ds})\in \nat_0^{ds}: \lfloor b^{l_j-1}\rfloor\le k_j< b^{l_j}\ \mathrm{for}\ 1\le j\le ds\} .
  \end{align}
In an analogous manner to $\Dcal_d$, we define a digital interlacing function $\mathcal{E}_d$ for non-negative integers. For $k_1,\ldots, k_{ds}\in \nat_0$, we represent the $b$-adic expansion of $k_j$ by $k_j=\kappa_{j,0}+\kappa_{j,0}b+\cdots $ for $1\le j\le ds$, where $\kappa_{j,u} \in \FF_b$ (and where $\kappa_{j,u}=0$ for all $u$ large enough). Then, $\mathcal{E}_d$ denotes the following mapping from $(k_1,\ldots, k_{ds})\in \nat_0^{ds}$ to $(k'_1,\ldots, k'_{s})\in \nat_0^s$, where
  \begin{align*}
    k'_j=\sum_{a=0}^{\infty}\sum_{r=1}^{d}\kappa_{(j-1)d+r,a}b^{r-1+ad} ,
  \end{align*}
for $1\le j\le s$.
Then we define the following sum of the Walsh coefficients of $f$ over $\bsk\in \mathcal{B}_{d,\bsl,s}$,
  \begin{align*}
     \sigma^2_{d,\bsl,s}(f)=\sum_{\bsk\in \BB_{d,\bsl,s}}|\hat{f}(\mathcal{E}_d(\bsk))|^2 ,
  \end{align*}
and we introduce
  \begin{align*}
     \Gamma_{\bsl,d}(\bsq, p)=\frac{1}{b^{2m}}\sum_{n,n^\prime =0}^{b^m-1}\prod_{j=1}^{ds} \EE\left[ \mathrm{wal}_{k_j}(\Pi_j(z_{n,j})\ominus \Pi_j(z_{n^\prime ,j}))\right] ,
  \end{align*}
where $\bsk = (k_1, k_2,\ldots, k_{ds}) \in \mathcal{B}_{d,\bsl,s}$ is an arbitrary element and the operator $\ominus$ denotes the digitwise subtraction modulo $b$, that is, for $x,y\in [0,1)$ with $b$-adic expansions $x=\sum_{i=1}^{\infty}x_ib^{-i}$ and $y=\sum_{i=1}^{\infty}y_ib^{-i}$, $\ominus$ is defined as
  \begin{align*}
     x\ominus y=\sum_{i=1}^{\infty}\frac{z_i}{b^i} ,
  \end{align*}
where $z_i=x_i-y_i \pmod b$. We note that $\Gamma_{\bsl,d}(\bsq, p)$ is independent of the choice of $\bsk\in \mathcal{B}_{d,\bsl,s}$, and depends only on the point set $\{\bsz_0,\ldots, \bsz_{b^m-1}\}$, see \cite{D11}. According to \cite[Lemma~7]{D11}, we have
  \begin{align}
     \Var[\hat{I}(f)] = \sum_{\bsl\in \nat_0^{ds}\setminus\{\bszero\}}\sigma^2_{d,\bsl,s}(f)\Gamma_{\bsl,d}(\bsq, p).
  \label{eq:variance}
  \end{align}

By applying the property of polynomial lattice rules to this expression of $\Var[\hat{I}(f)]$, we obtain the following lemma.

\begin{lemma} \label{lemma:variance}
Let $d\in \nat$ and $f\in L_2([0,1]^s)$. Let the estimator $\hat{I}$ be given by
  \begin{align*}
    \hat{I}(f) = \frac{1}{b^m}\sum_{n=0}^{b^m-1}f(\bsy_n),
  \end{align*}
where $\{\bsy_0,\ldots, \bsy_{b^m-1}\}$ is an interlaced scrambled polynomial lattice point set with generating vector $\bsq$ and modulus $p$. Then, we have
  \begin{align}
     \Var[\hat{I}(f)] = \sum_{\emptyset\ne u\subseteq \{1:ds\}}\frac{b^{|u|}}{(b-1)^{|u|}}\sum_{\bsl_u\in \nat^{|u|}}\frac{\sigma^2_{d,(\bsl_u,\bszero),s}(f)}{b^{|\bsl_u|_1}}\sum_{\bsk\in \mathcal{B}_{d,(\bsl_u,\bszero),s}\cap D^{\perp }}1 ,
  \label{eq:lemma_variance}
  \end{align}
where $|\bsl_u|_1=\sum_{j\in u}l_j$ and $D^\perp$ is the dual polynomial lattice for the polynomial lattice point set with generating vector $\bsq$ and modulus $p$.
\end{lemma}

\begin{proof}
This follows immediately from \cite[Corollary~13.7]{DP10}.
\end{proof}

\subsection{A bound on the Walsh coefficients}\label{ssec:var}

Below we define a variation $V_{\alpha}^{(s)}(f)$ of order $\alpha \ge 1$ for functions $f:[0,1]^s\to\mathbb{R}$. See  \cite[pp.~1386, 1387]{D11} for a derivation of this definition. In particular, in \cite{D11} it is shown that if the partial derivatives $\frac{\partial^{\alpha_1+\cdots + \alpha_s} f}{\partial x_1^{\alpha_1} \cdots \partial x_s^{\alpha_s}}$ are continuous for a given $\boldsymbol{\alpha}=(\alpha_1,\ldots, \alpha_s) \in \{1:\alpha\}^s$, then
\begin{equation*}
V_{\boldsymbol{\alpha}}^{(s)}(f) = \left( \int_{[0,1]^s} \left| \frac{\partial^{\alpha_1+\cdots + \alpha_s} f}{\partial x_1^{\alpha_1} \cdots \partial x_s^{\alpha_s}} \right|^2 \rd \bsx \right)^{1/2}.
\end{equation*}

To define the variation $V_{\alpha}^{(s)}$, let $J = \prod_{j=1}^{\alpha s} [a_j b^{-l_j}, (a_j+1) b^{-l_j})$, where $0 \le a_j < b^{l_j}$ and $l_j \in \mathbb{N}$ for $1 \le j \le \alpha s$. The set $\mathcal{D}_\alpha(J) = \{\Dcal_\alpha(\bsx): \bsx \in J\}$ is the product of a union of intervals except for a countable number of points (see \cite{D11}). Let $\boldsymbol{\alpha} \in \{1:\alpha\}^s$. For $\bst \in [0,1)^s$ and $\bsx_1,\ldots, \bsx_s \in (-1,1)^s$ we define the difference operator
\begin{align*}
\Delta_{\bsalpha}(\bst; \bsx_1,\ldots, \bsx_s) f  = & \sum_{v_1 \subseteq \{1:\alpha_1\}} \cdots \sum_{v_s \subseteq \{1:\alpha_s\}} (-1)^{|v_1| + \cdots + |v_s|}  \\ & \times f\left(t_1 + \sum_{i_1 \in v_1} x_{1,i_1}, \ldots, t_s + \sum_{i_s \in v_s} x_{s, i_s}\right).
\end{align*}
Then we define the generalized Vitali variation
\begin{equation*}
V_{\boldsymbol{\alpha}}^{(s)}(f) = \sup_{\mathcal{P}} \left(\sum_{J \in \mathcal{P}} \mathrm{Vol}(\mathcal{D}_\alpha(J)) \sup \left|\frac{\Delta_{\boldsymbol{\alpha}}(\bst; \bsx_1,\ldots, \bsx_s) f}{\prod_{j=1}^s \prod_{r=1}^{\alpha_j} x_{j,r} }\right|^2 \right)^{1/2},
\end{equation*}
where the first supremum $\sup_{\mathcal{P}}$ is over all partitions of $[0,1)^{\alpha s}$ into subcubes of the form $J = \prod_{j=1}^{\alpha s} [ a_j b^{-l_j}, (a_j+1) b^{-l_j})$ with $0 \le a_j < b^{l_j}$ and $l_j \in \mathbb{N}$ for $1 \le j \le \alpha s$, and the second supremum is taken over all $\bst \in \mathcal{D}_\alpha(J)$ and $\bsx_j = (x_{j,1},\ldots, x_{j,\alpha_j})$ with $x_{j,r} = \tau_{j,r} b^{-\alpha(l_j-1)-r}$ where $\tau_{j,r} \in \{1-b, \ldots, b-1\} \setminus \{0\}$ for $1 \le r \le \alpha_j$ and $1 \le j \le s$ and such that for all the points at which $f$ is evaluated in $\Delta_{\boldsymbol{\alpha}}(\bst; \bsx_1,\ldots, \bsx_s)$ are in $\mathcal{D}_\alpha(\prod_{j=1}^{\alpha s} [ b^{-l_j+1} \lfloor a_j/b\rfloor, b^{-l_j+1} \lfloor a_j/b \rfloor + 1))$.

For $ \emptyset \neq u \subseteq \{1:s\}$ let $|u|$ denote the number of elements in $u$ and let $V_{\boldsymbol{\alpha}}^{(|u|)}(f_u; u)$ denote the generalized Vitali variation with coefficient $\boldsymbol{\alpha} \in \{1:\alpha\}^{|u|}$ of the $|u|$-dimensional function
$$f_u(\bsx_u) = \int_{[0,1]^{s-|u|}} f(\bsx) \rd \bsx_{\{1:s\}\setminus u}.$$ For $u = \emptyset$ we set $f_\emptyset = \int_{[0,1]^s} f(\bsx) \rd\bsx$ and we define $V_{\boldsymbol{\alpha}}^{(0)}(f_\emptyset; \emptyset) = |f_{\emptyset}|$.
Now we define the generalized weighted Hardy and Krause variation of $f$ of order $\alpha$ by (cf. \cite[p. 1387]{D11})
\begin{equation*}
V_{\alpha, \bsgamma}(f) = \left(\sum_{u \subseteq \{1:s\}} \gamma_u^{-1} \sum_{\bsalpha \in \{1:\alpha\}^{|u|}} (V_{\boldsymbol{\alpha}}^{(|u|)}(f_u; u))^2 \right)^{1/2},
\end{equation*}
where $(\gamma_u)_{u \in U}$ is a sequence of nonnegative real numbers and $U = \{u \subset \mathbb{N}: |u| < \infty\}$.

Let $f:[0,1]^s \to \mathbb{R}$ and let
\begin{equation*}
f(\bsx) = \sum_{u \subseteq \{1:s\}} g_u(\bsx_u)
\end{equation*}
denote the ANOVA decomposition of $f$, that is, $g_\emptyset = \int_{[0,1]^s} f(\bsx)\rd \bsx$ and
\begin{equation*}
g_u(\bsx_u) = f_u - \sum_{v \subset u} g_v,
\end{equation*}
where $v \subset u$ means that $v$ is a proper subset of $u$. We have $\int_0^1 g_u(\bsx_u) \rd x_j = 0$ for $j \in u$ and $\frac{\partial g_u}{\partial x_j} = 0$ for $j \notin u$. Then we have
\begin{equation*}
V_{\alpha, \bsgamma}(g_u) = \gamma_u^{-1/2} V_\alpha^{(|u|)}(g_u; u).
\end{equation*}

Let $\bsl=(l_1,\ldots, l_{ds})\in \nat_0^{ds}$ and let $u=\{i\in \{1:ds\}: l_i>0\}$. Then we denote $\bsl$ by $(\bsl_u,\bszero)$. Let $v(u) \subseteq \{1:s\}$ denote the set of $1 \le i \le s$ such that $u \cap \{(i-1)d+1:id\} \neq \emptyset$. Then
\begin{equation*}
\sigma_{d,(\bsl_u, \bszero),s}(g_v) = 0 \quad \mbox{if } v \neq v(u).
\end{equation*}
Thus we have
\begin{equation*}
\sigma_{d, (\bsl_u,\bszero),s}(f) = \sigma_{d,(\bsl_u,\bszero),s}(g_{v(u)}).
\end{equation*}

Using \cite[Lemma~9]{D11}, we therefore obtain
\begin{equation*}
\sigma_{d, (\bsl_u, \bszero), s}(f) = \sigma_{d,(\bsl_u,\bszero),s}(g_{v(u)}) \le 2^{|v(u)| \max(d-\alpha,0)} \beta(\bsl_u,\bszero) \sqrt{\gamma_{v(u)}} V_{\alpha, \bsgamma}(f),
\end{equation*}
where the definition of $\beta(\bsl_v,\bszero)$ is given as follows. Let $u_i=u\cap \{(i-1)d+1: id\}$ and $\alpha_i=\min(\alpha,|u_i|)$ for $i\in v(u)$. Let $\beta'_j=(b-1)b^{-j+(i-1)d-(l_j-1)d}$ for $j\in u_i$ and $1\le i\le s$. Let $\beta_{i,1}<\beta_{i,2}<\cdots < \beta_{i,|u_i|}$ for $i\in v(u)$ be such that $\{\beta_{i,1},\ldots, \beta_{i,|u_i|}\}=\{\beta'_j: j\in u_i\}$, that is $\{\beta_{i,j}: 1\le j\le |u_i|\}$ is just a reordering of the elements of the set $\{\beta'_j: j\in u_i\}$. We define $\beta(\bsl_u,\bszero)$ as
\begin{equation*}
\beta(\bsl_u,\bszero)=\prod_{i\in v(u)}\prod_{j=1}^{\alpha_i}\beta_{i,j} .
\end{equation*}
In the following lemma we provide a bound on the coefficients $\beta(\bsl_u, \bszero)$.

\begin{lemma}\label{lemma:gamma_bound}
Let $\alpha, d, s\in \nat$. For any $\emptyset \ne u\subseteq \{1:s\}$ and $(\bsl_u,\bszero)\in \nat_0^{ds}$ such that $l_j>0$ for $j\in u$, let $\beta(\bsl_u,\bszero)$ be given as above. Then we have
\begin{equation*}
\beta(\bsl_u,\bszero) \le b^{(2d-1)\alpha |v(u)|/2}\prod_{j\in u}(b-1)b^{-\min(\alpha,d)l_j-\alpha/2}.
\end{equation*}
\end{lemma}

\begin{proof}
First, we consider the case $d\le \alpha$. Since $|u_i|\le d\le \alpha$ in this case, we always have $\alpha_i=|u_i|$. Thus, from the definition of $\beta(\bsl_u,\bszero)$, we have
  \begin{align}
     \beta(\bsl_u,\bszero) & = \prod_{i\in v(u)}\prod_{j\in u_i}(b-1)b^{-j+(i-1)d-(l_j-1)d} \nonumber \\
     & \le \prod_{i\in v(u)} \left( \prod_{j\in u_i}(b-1)b^{-dl_j} \right) \left( \prod_{h=1}^{|u_i|}b^{d-h} \right) \nonumber \\
     & = \prod_{i\in v(u)}b^{d|u_i|-|u_i|(|u_i|+1)/2}\prod_{j\in u_i}(b-1)b^{-dl_j} \nonumber \\
     & \le \prod_{i\in v(u)}b^{d\alpha-\alpha(|u_i|+1)/2}\prod_{j\in u_i}(b-1)b^{-dl_j} \nonumber \\
     & = b^{(2d-1)\alpha|v(u)|/2}\prod_{j\in u}(b-1)b^{-dl_j-\alpha/2} .
  \label{eq:gamma_alpha_d}
  \end{align}

Next, we consider the case $d>\alpha$. Since $0<\beta_{i,1}<\beta_{i,2}<\cdots < \beta_{i,|u_i|}$ for every $i\in v(u)$, it holds that
\begin{equation*}
     \prod_{j=1}^{\alpha_i}\beta_{i,j} \le \left( \prod_{j=1}^{|u_i|}\beta_{i,j} \right)^{\alpha_i/|u_i|} = \prod_{j=1}^{|u_i|}\beta_{i,j}^{\alpha_i/|u_i|} .
\end{equation*}
Thus we have
  \begin{align}
     \beta(\bsl_u,\bszero) & \le \prod_{i\in v(u)}\prod_{j=1}^{|u_i|}\beta_{i,j}^{\alpha_i/|u_i|} \nonumber \\
     & \le \prod_{i\in v(u)}\prod_{j\in u_i}(b-1)^{\alpha_i/|u_i|}b^{-d\alpha_i l_j/|u_i|}\prod_{h=1}^{|u_i|}b^{(d-h)\alpha_i/|u_i|} \nonumber \\
     & = \prod_{i\in v(u)}b^{d\alpha_i -\alpha_i (|u_i|+1)/2}\prod_{j\in u_i}(b-1)^{\alpha_i/|u_i|}b^{-d\alpha_i l_j/|u_i|}.
  \label{eq:gamma_d_alpha}
  \end{align}
Since $\alpha_i=\min(\alpha,|u_i|)$, we have $\alpha_i/|u_i|\le 1$ and $d\alpha_i/|u_i|\ge \alpha$. The latter inequality is obtained as follows: If $\alpha < |u_i| \le d$, then $\alpha_i=\alpha$ and $d\alpha_i/|u_i|= d\alpha/|u_i|\ge \alpha$. Otherwise if $|u_i| \le \alpha < d$, then $\alpha_i=|u_i|$ and $d\alpha_i/|u_i|= d > \alpha$.

Applying the inequalities $\alpha_i/|u_i|\le 1$ and $d\alpha_i/|u_i|\ge \alpha$ to (\ref{eq:gamma_d_alpha}), we have
  \begin{align}
     \beta(\bsl_u,\bszero) & \le \prod_{i\in v(u)}b^{d\alpha -\alpha (|u_i|+1)/2}\prod_{j\in u_i}(b-1)b^{-\alpha l_j} \nonumber \\
     & = b^{(2d-1)\alpha|v(u)|/2}\prod_{j\in u}(b-1)b^{-\alpha l_j-\alpha/2} .
  \label{eq:gamma_d_alpha_2}
  \end{align}
Combining (\ref{eq:gamma_alpha_d}) and (\ref{eq:gamma_d_alpha_2}), the result follows.
\end{proof}

\subsection{A bound on the variance}\label{sec_bound_var}

Using Lemma \ref{lemma:variance} and the bound on the Walsh coefficients given in the previous subsection, we can now prove a bound on the variance of the estimator. We can then use this bound to introduce a quality criterion for the construction of interlaced scrambled polynomial lattice rules.

Let us define
  \begin{align*}
     r_{\alpha,d}(k):=\left\{ \begin{array}{ll}
     1 & \mathrm{if}\ k=0, \\
     (b-1)b^{-(2\min(\alpha,d)+1)\mu(k)-\alpha+1} & \mathrm{otherwise}, \\
     \end{array} \right.
  \end{align*}
where we introduce the weight $\mu(k):=a$ for $k=\kappa_0+\kappa_1 b+\cdots + \kappa_{a-1} b^{a-1}$ such that $\kappa_{a-1}\ne 0$. 
For $\bsk=(k_1,\ldots, k_{ds})\in \nat_0^{ds}$, let $r_{\alpha,d}(\bsk)=\prod_{j=1}^{ds}r_{\alpha,d}(k_j)$. The following corollary gives a bound on the variance of the estimator.

\begin{corollary}\label{corollary:bound_variance}
Let $\alpha,d \in \nat$. Let $f:[0,1]^s\to \RR$ satisfy $V_{\alpha, \bsgamma}(f)< \infty$. Let the estimator $\hat{I}$ be given by
  \begin{align*}
    \hat{I}(f) = \frac{1}{b^m}\sum_{n=0}^{b^m-1}f(\bsy_n),
  \end{align*}
where $\{\bsy_0,\ldots, \bsy_{b^m-1}\}$ is an interlaced scrambled polynomial lattice point set of order $d \ge 1$ with generating vector $\bsq \in (\mathbb{F}_b[x])^{ds}$ and modulus $p$. Then we have
  \begin{align*}
     \Var[\hat{I}(f)] \le V^2_{\alpha, \bsgamma}(f)\sum_{\emptyset\ne u\subseteq \{1:ds\}} \gamma_{v(u)}D_{\alpha,d}^{|v(u)|}\sum_{\substack{\bsk_u\in \nat^{|u|}\\ (\bsk_u,\bszero)\in D^{\perp }}}r_{\alpha,d}(\bsk_u,\bszero) ,
  \end{align*}
where $D^\perp$ is a dual polynomial lattice of the polynomial lattice rule with generating vector $\bsq$ and modulus $p$ as in Definition \ref{def:dual_net}, $v(u) \subseteq \{1:s\}$ is the set of all $i \in \{1:s\}$ such that $u \cap \{(i-1)d+1:id\} \neq \emptyset$, and $D_{\alpha,d}:=4^{\max(d-\alpha,0)}b^{(2d-1)\alpha}$.
\end{corollary}

\begin{proof}
From the bound on the Walsh coefficients given in the previous subsection and Lemma \ref{lemma:gamma_bound}, we have
  \begin{align*}
     \sigma^2_{d,(\bsl_u,\bszero),s}(f) & \le V^2_{\alpha, \bsgamma}(f) \gamma_{v(u)} D_{\alpha,d}^{|v(u)|}\prod_{j\in u}(b-1)^2b^{-2\min(\alpha,d) l_j-\alpha} \\
     & = V^2_{\alpha, \bsgamma}(f) \gamma_{v(u)} D_{\alpha,d}^{|v(u)|}\frac{(b-1)^{2|u|}}{b^{2\min(\alpha,d) |\bsl_u|_1+\alpha |u|}} .
  \end{align*}
We note that it holds that $\mu(k_j)=l_j$ for all $(\bsk_u,\bszero) \in \mathcal{B}_{d,(\bsl_u,\bszero),s}$. Inserting the above inequality into (\ref{eq:lemma_variance}), we have
  \begin{align*}
     \Var[\hat{I}(f)] & \le V^2_{\alpha, \bsgamma}(f) \sum_{\emptyset\ne u\subseteq \{1:ds\}}\gamma_{v(u)} D_{\alpha,d}^{|v(u)|}\sum_{\bsl_u\in \nat^{|u|}}\frac{(b-1)^{|u|}}{b^{(2\min(\alpha,d)+1) |\bsl_u|_1+\alpha |u|-|u|}}\sum_{\bsk\in \mathcal{B}_{d,(\bsl_u,\bszero),s}\cap D^{\perp }}1 \\
     & = V^2_{\alpha, \bsgamma}(f)\sum_{\emptyset\ne u\subseteq \{1:ds\}} \gamma_{v(u)}D_{\alpha,d}^{|v(u)|}\sum_{\substack{\bsk_u\in \nat^{|u|}\\ (\bsk_u,\bszero)\in D^{\perp }}}r_{\alpha,d}(\bsk_u,\bszero) .
  \end{align*}
\end{proof}

We denote the double sum in Corollary \ref{corollary:bound_variance} by
  \begin{align}
     B_{\alpha,d, \bsgamma}(\bsq, p) := \sum_{\emptyset\ne u\subseteq \{1:ds\}} \gamma_{v(u)}D_{\alpha,d}^{|v(u)|}\sum_{\substack{\bsk_u\in \nat^{|u|}\\ (\bsk_u,\bszero)\in D^{\perp }}}r_{\alpha,d}(\bsk_u,\bszero) .
  \label{eq:criterion}
  \end{align}
This value depends on the smoothness $\alpha$, the weights $(\gamma_u)_{u\subseteq \{1:s\}}$, both of which come from the function space, the interlacing factor $d$ and the polynomial lattice rule with $ds$ components. We note that it is independent of a particular function $f$. Thus, it is possible to use $B_{\alpha,d, \bsgamma}(\bsq, p)$ as a quality criterion for the construction of interlaced scrambled polynomial lattice rules. The following lemma gives a more computable form of $B_{\alpha,d, \bsgamma}(\bsq, p)$.

\begin{lemma}\label{lemma:criterion}
For $z\in [0,1)$, let
  \begin{align*}
     \phi_{\alpha,d}(z):= \frac{(b-1)\left(b-1-b^{2\min(\alpha,d)\lfloor \log_b z\rfloor}(b^{2\min(\alpha,d)+1}-1)\right)}{b^{\alpha}(b^{2\min(\alpha,d)}-1)} ,
  \end{align*}
where we set $b^{2\min(\alpha,d)\lfloor \log_b 0\rfloor}=0$. Let $B_{\alpha,d, \bsgamma}(\bsq, p)$ be given by (\ref{eq:criterion}). Then, we have
  \begin{align*}
     B_{\alpha,d, \bsgamma}(\bsq, p) = \frac{1}{b^m}\sum_{n=0}^{b^m-1}\sum_{\emptyset\ne v\subseteq \{1:s\}} \gamma_{v}D_{\alpha,d}^{|v|}\prod_{j\in v}\Big[ -1+\prod_{k=1}^{d}\left(1+\phi_{\alpha,d}(z_{n,(j-1)d+k})\right)\Big] .
  \end{align*}
In particular, for product weights $\gamma_v=\prod_{j\in v}\gamma_j$, we have
  \begin{align*}
     B_{\alpha,d, \bsgamma}(\bsq, p) = -1+\frac{1}{b^m}\sum_{n=0}^{b^m-1}\prod_{j=1}^{s} \Big[ 1- \gamma_j D_{\alpha,d} + \gamma_jD_{\alpha,d}\prod_{k=1}^{d}\left(1+\phi_{\alpha,d}(z_{n,(j-1)d+k})\right)\Big] .
  \end{align*}
\end{lemma}

\begin{proof}
Using Lemma \ref{lamma:dual_walsh}, we can rewrite (\ref{eq:criterion}) as follows
  \begin{align}
      B_{\alpha,d, \bsgamma}(\bsq, p) \nonumber & = \sum_{\emptyset\ne u\subseteq \{1:ds\}}\gamma_{v(u)}D_{\alpha,d}^{|v(u)|}\sum_{\bsk_u\in \nat^{|u|}}r_{\alpha,d}(\bsk_u,\bszero)\frac{1}{b^m}\sum_{n=0}^{b^m-1}\mathrm{wal}_{(\bsk_u,\bszero)}(\bsz_n) \nonumber \\
      & = \frac{1}{b^m}\sum_{n=0}^{b^m-1}\sum_{\emptyset\ne u\subseteq \{1:ds\}} \gamma_{v(u)}D_{\alpha,d}^{|v(u)|}\sum_{\bsk_u\in \nat^{|u|}}r_{\alpha,d}(\bsk_u,\bszero)\mathrm{wal}_{(\bsk_u,\bszero)}(\bsz_n) \nonumber \\
      & = \frac{1}{b^m}\sum_{n=0}^{b^m-1}\sum_{\emptyset\ne u\subseteq \{1:ds\}} \gamma_{v(u)}D_{\alpha,d}^{|v(u)|}\prod_{j\in u}\left[\sum_{k_j=1}^{\infty}r_{\alpha,d}(k_j)\mathrm{wal}_{k_j}(z_{n,j})\right] .
  \label{eq:criterion2}
  \end{align}
Following along similar lines as in the proof of \cite[Theorem~7.3]{B10}, we have for $z\in [0,1)$
  \begin{align*}
     \sum_{k=1}^{\infty}r_{\alpha,d}(k)\mathrm{wal}_{k}(z) & = (b-1)b^{-\alpha+1}\sum_{l=1}^{\infty}\frac{1}{b^{(2\min(\alpha,d)+1)l}}\sum_{k=b^{l-1}}^{b^l-1}\mathrm{wal}_{k}(z) \\
     & = (b-1)b^{-\alpha+1}\cdot \frac{b-1-b^{2\min(\alpha,d)\lfloor \log_b z\rfloor}(b^{2\min(\alpha,d)+1}-1)}{b(b^{2\min(\alpha,d)}-1)} \\
     & = \phi_{\alpha,d}(z).
  \end{align*}
Thus, the bracket in (\ref{eq:criterion2}) reduces to $\phi_{\alpha,d}(z_{n,j})$. We further rearrange \eqref{eq:criterion2}. For a given $\emptyset \ne w \subseteq \{1:s\}$ we now consider sets $\emptyset \ne u \subseteq \{1:ds\}$ such that $v(u) = w$. Then $u$ has to contain at least one element from $\{(j-1)d+1:jd\}$ for any $j\in w$. We therefore obtain
  \begin{align*}
     B_{\alpha,d, \bsgamma}(\bsq, p) = \frac{1}{b^m}\sum_{n=0}^{b^m-1}\sum_{\emptyset\ne w\subseteq \{1:s\}} \gamma_{w}D_{\alpha,d}^{|w|}\prod_{j\in w}\Big[ -1+\prod_{k=1}^{d}\left(1+\phi_{\alpha,d}(z_{n,(j-1)d+k})\right)\Big] .
  \end{align*}
In case of the product weights $\gamma_{v}=\prod_{j\in v}\gamma_j$, the last expression can be further simplified into
  \begin{align*}
     B_{\alpha,d, \bsgamma}(\bsq, p)  & = \frac{1}{b^m}\sum_{n=0}^{b^m-1}\sum_{\emptyset\ne w\subseteq \{1:s\}}\prod_{j\in w}\gamma_{j}D_{\alpha,d}\Big[ -1+\prod_{k=1}^{d}\left(1+\phi_{\alpha,d}(z_{n,(j-1)d+k})\right)\Big] \\
     & = -1+\frac{1}{b^m}\sum_{n=0}^{b^m-1}\prod_{j=1}^{s}\Big[ 1-\gamma_j D_{\alpha,d} + \gamma_j D_{\alpha,d}\prod_{k=1}^{d}\left(1+\phi_{\alpha,d}(z_{n,(j-1)d+k})\right)\Big] .
  \end{align*}
Hence the result follows.
\end{proof}

\section{Component-by-component construction of polynomial lattice point sets}
\label{CBC construction}

\subsection{Convergence rate of the variance}\label{subsec:proofthm1}

In the proof of Theorem \ref{theorem:cbc_proof} and its subsequent remark, we shall use Jensen's inequality, which states that for a sequence $(a_k)$ of non-negative real numbers we have
	\begin{align*}
		\left( \sum a_k\right)^{\lambda} \le \sum a_k^{\lambda},
	\end{align*}
for any $0<\lambda\le 1$.

\begin{proof}[Proof of Theorem~\ref{theorem:cbc_proof}]
We prove the result by following along the same lines as in the proof of \cite[Theorem~4.4]{DKPS05}. We proceed by induction. First for $\tau=1$, that is, for $j_0=1$ and $d_0=1$, we have
	\begin{align*}
	B_{\alpha,d, \bsgamma}(1,p) & = \gamma_{\{1\}}D_{\alpha,d} \sum_{\substack{ k=1\\ b^m|k}}^{\infty}r_{\alpha,d}(k) \\
	& = \gamma_{\{1\}}D_{\alpha,d}(b-1)b^{-\alpha+1} \sum_{l=1}^{\infty}b^{-(2\min(\alpha,d)+1)l}\sum_{\substack{ k=b^{l-1}\\ b^m|k}}^{b^l-1}1 \\
	& = \gamma_{\{1\}}D_{\alpha,d}(b-1)b^{-\alpha+1} \sum_{l=m+1}^{\infty}(b^l-b^{l-1})b^{-m}b^{-(2\min(\alpha,d)+1)l} \\
	& = \frac{1}{b^{(2\min(\alpha,d)+1)m}}\gamma_{\{1\}}D_{\alpha,d} \frac{(b-1)^2}{b^{\alpha}(b^{2\min(\alpha,d)}-1)} \\
	& \le \frac{1}{(b^m-1)^{1/\lambda}} \left[ \gamma_{\{1\}}^{\lambda}D_{\alpha,d}^{\lambda} \left(\frac{(b-1)^2}{b^{\alpha}(b^{2\min(\alpha,d)}-1)}\right)^{\lambda}\right]^{\frac{1}{\lambda}} ,
	\end{align*}
for all $1/(2\min(\alpha,d)+1)< \lambda \le 1$. Consequently, we obtain
	\begin{align*}
	B_{\alpha,d, \bsgamma}(1,p) \le \frac{1}{(b^m-1)^{1/\lambda}} \left[ \gamma_{\{1\}}^{\lambda}C_{\alpha,d,\lambda,1} \right]^{1/\lambda} .
	\end{align*}

Next, suppose that for some $\tau=(j_0-1)d+d_0$ such that $j_0, d_0 \in \nat$ and $0< d_0\le d$, we have $\bsq_{\tau}\in R_{b,m}^{\tau}$ which satisfies
	\begin{align*}
	B_{\alpha,d, \bsgamma}(\bsq_{\tau},p) \le \frac{1}{(b^m-1)^{1/\lambda}} \left[ \sum_{\emptyset \ne u\subseteq \{1:j_0-1\}}\gamma_u^\lambda C_{\alpha,d,\lambda,d}^{|u|}+C_{\alpha,d,\lambda,d_0}\sum_{u\subseteq \{1:j_0-1\}}\gamma_{u\cup\{j_0\}}^\lambda C_{\alpha,d,\lambda,d}^{|u|}\right]^{1/\lambda} .
	\end{align*}
We denote $\tau+1=(j_1-1)d+d_1$ such that $j_1, d_1 \in \nat$ and $0< d_1\le d$. It is obvious that we have
	\begin{align*}
	(j_1,d_1)=\left\{ \begin{array}{ll}
     (j_0+1,1) & \mathrm{if}\ d_0=d, \\
     (j_0,d_0+1) & \mathrm{otherwise}. \\
     \end{array} \right.
	\end{align*}
Now we consider from (\ref{eq:criterion})
	\begin{align*}
	B_{\alpha,d, \bsgamma}((\bsq_{\tau}, \tilde{q}_{\tau+1}), p) & = \sum_{\emptyset\ne u\subseteq \{1:\tau+1\}} \gamma_{v(u)}D_{\alpha,d}^{|v(u)|}\sum_{\substack{\bsk_u\in \nat^{|u|}\\ (\bsk_u,\bszero)\in D^{\perp }}}r_{\alpha,d}(\bsk_u,\bszero) \\
	& = \sum_{\emptyset\ne u\subseteq \{1:\tau\}} \gamma_{v(u)}D_{\alpha,d}^{|v(u)|}\sum_{\substack{\bsk_u\in \nat^{|u|}\\ (\bsk_u,\bszero)\in D^{\perp }}}r_{\alpha,d}(\bsk_u,\bszero) \\
	& \quad + \sum_{u\subseteq \{1:\tau\}} \gamma_{v(u\cup\{\tau+1\})}D_{\alpha,d}^{|v(u\cup\{\tau+1\})|}\sum_{\substack{\bsk_{u\cup\{\tau+1\}}\in \nat^{|u|+1}\\ (\bsk_{u\cup\{\tau+1\}},\bszero)\in D^{\perp }}}r_{\alpha,d}(\bsk_{u\cup\{\tau+1\}},\bszero) \\
	& = B_{\alpha,d, \bsgamma}(\bsq_{\tau}, p)+\theta(\tilde{q}_{\tau+1}) ,
	\end{align*}
where we define
	\begin{align*}
	\theta(\tilde{q}_{\tau+1}) := \sum_{u\subseteq \{1:\tau\}} \gamma_{v(u\cup\{\tau+1\})}D_{\alpha,d}^{|v(u\cup\{\tau+1\})|}\sum_{\substack{\bsk_{u\cup\{\tau+1\}}\in \nat^{|u|+1}\\ (\bsk_{u\cup\{\tau+1\}},\bszero)\in D^{\perp }}}r_{\alpha,d}(\bsk_{u\cup\{\tau+1\}},\bszero) .
	\end{align*}
In order to minimize $B_{\alpha,d, \bsgamma}((\bsq_{\tau}, \tilde{q}_{\tau+1}), p)$ as a function $\tilde{q}_{\tau+1}$, we only need to consider $\theta(\tilde{q}_{\tau+1})$. Based on an averaging argument, that is, the minimum value of $\theta(\tilde{q}_{\tau+1})$ is less than or equal to the average value of $\theta(\tilde{q}_{\tau+1})$ over $\tilde{q}_{\tau+1}\in R_{b,m}$, we have for $1/(2\min(\alpha,d)+1)<\lambda\le 1$
	\begin{align}
	\theta^{\lambda}(q_{\tau+1}) & \le \frac{1}{b^m-1}\sum_{\tilde{q}_{\tau+1}\in R_{b,m}}\theta^{\lambda}(\tilde{q}_{\tau+1}) \nonumber \\
	& \le \frac{1}{b^m-1}\sum_{\tilde{q}_{\tau+1}\in R_{b,m}}\sum_{u\subseteq \{1:\tau\}} \gamma_{v(u\cup\{\tau+1\})}^{\lambda}D_{\alpha,d}^{\lambda|v(u\cup\{\tau+1\})|}\nonumber \\ & \quad \times \sum_{\substack{\bsk_{u\cup\{\tau+1\}}\in \nat^{|u|+1}\\ (\bsk_{u\cup\{\tau+1\}},\bszero)\in D^{\perp }}}r_{\alpha,d}^{\lambda}(\bsk_{u\cup\{\tau+1\}},\bszero) \nonumber \\
	& = \sum_{u\subseteq \{1:\tau\}}\gamma_{v(u\cup\{\tau+1\})}^{\lambda}D_{\alpha,d}^{\lambda|v(u\cup\{\tau+1\})|} \nonumber \\ & \quad \times \frac{1}{b^m-1}\sum_{\tilde{q}_{\tau+1}\in R_{b,m}}\sum_{\substack{\bsk_{u\cup\{\tau+1\}}\in \nat^{|u|+1}\\ (\bsk_{u\cup\{\tau+1\}},\bszero)\in D^{\perp }}}r_{\alpha,d}^{\lambda}(\bsk_{u\cup\{\tau+1\}},\bszero) ,
	\label{eq:theta}
	\end{align}
where we have used Jensen's inequality in the second inequality. For a fixed $u\subseteq \{1:\tau\}$ of the outermost sum in (\ref{eq:theta}), if $k_{\tau+1}$ is a multiple of $b^m$, we always have $\rtr_m(k_{\tau+1})=0$ and the corresponding term becomes independent of $\tilde{q}_{\tau+1}$, or otherwise we have $\rtr_m(k_{\tau+1})\ne 0$ and $\rtr_m(k_{\tau+1})\tilde{q}_{\tau+1}$ cannot be a multiple of $p$ by considering that $p$ is irreducible. Hence we have
	\begin{align}
	& \frac{1}{b^m-1}\sum_{\tilde{q}_{\tau+1}\in R_{b,m}}\sum_{\substack{\bsk_{u\cup\{\tau+1\}}\in \nat^{|u|+1}\\ (\bsk_{u\cup\{\tau+1\}},\bszero)\in D^{\perp }}}r_{\alpha,d}^{\lambda}(\bsk_{u\cup\{\tau+1\}},\bszero) \nonumber \\
	= & \sum_{\substack{k_{\tau+1}=1\\ b^m\mid k_{\tau+1}}}^{\infty}r_{\alpha,d}^{\lambda}(k_{\tau+1})\sum_{\substack{\bsk_u\in \nat^{|u|}\\ \rtr_m(\bsk_u)\cdot \bsq_u=0 \pmod{p}}}r_{\alpha,d}^{\lambda}(\bsk_u) \nonumber \\
	& + \frac{1}{b^m-1}\sum_{\substack{k_{\tau+1}=1\\ b^m\nmid k_{\tau+1}}}^{\infty}r_{\alpha,d}^{\lambda}(k_{\tau+1})\sum_{\substack{\bsk_u\in \nat^{|u|}\\ \rtr_m(\bsk_u)\cdot \bsq_u\ne 0 \pmod{p}}}r_{\alpha,d}^{\lambda}(\bsk_u) .
	\label{eq:theta-r1}
	\end{align}
For the first term of the right-hand side in (\ref{eq:theta-r1}), we have
	\begin{align*}
	\sum_{\substack{k_{\tau+1}=1\\ b^m\mid k_{\tau+1}}}^{\infty}r_{\alpha,d}^{\lambda}(k_{\tau+1}) & = (b-1)^{\lambda}b^{-\lambda(\alpha-1)}\sum_{l=1}^{\infty}b^{-(2\min(\alpha,d)+1)\lambda l}\sum_{\substack{k_{\tau+1}=b^{l-1}\\ b^m\mid k_{\tau+1}}}^{b^l-1}1 \\
	& = \frac{(b-1)^{1+\lambda}}{b^{m+1+\lambda(\alpha-1)}}\sum_{l=m+1}^{\infty}b^{(1-(2\min(\alpha,d)+1)\lambda )l} .
	\end{align*}
For the second term, on the other hand, we have
	\begin{align*}
	& \quad \frac{1}{b^m-1}\sum_{\substack{k_{\tau+1}=1\\ b^m\nmid k_{\tau+1}}}^{\infty}r_{\alpha,d}^{\lambda}(k_{\tau+1}) \\
	& = \frac{1}{b^m-1}\sum_{l=1}^{m}\sum_{\substack{k_{\tau+1}=b^{l-1}\\ b^m\nmid k_{\tau+1}}}^{b^l-1}r_{\alpha,d}^{\lambda}(k_{\tau+1})+ \frac{1}{b^m-1}\sum_{l=m+1}^{\infty}\sum_{\substack{k_{\tau+1}=b^{l-1}\\ b^m\nmid k_{\tau+1}}}^{b^l-1}r_{\alpha,d}^{\lambda}(k_{\tau+1}) \\
	& = \frac{(b-1)^{1+\lambda}}{(b^m-1)b^{1+\lambda(\alpha-1)}}\sum_{l=1}^{m}b^{(1-(2\min(\alpha,d)+1)\lambda) l} \\
	& \quad + \frac{(b-1)^{1+\lambda}}{b^{m+1+\lambda(\alpha-1)}}\sum_{l=m+1}^{\infty}b^{(1-(2\min(\alpha,d)+1)\lambda) l} \\
	& = \frac{(b-1)^{1+\lambda}}{b^{1+\lambda(\alpha-1)}}\left[ \frac{1}{b^m-1}\sum_{l=1}^{m}b^{(1-(2\min(\alpha,d)+1)\lambda) l}+ \frac{1}{b^m}\sum_{l=m+1}^{\infty}b^{(1-(2\min(\alpha,d)+1)\lambda) l}\right] .
	\end{align*}
By inserting these equalities into (\ref{eq:theta-r1}), we have
	\begin{align*}
	& \quad \frac{1}{b^m-1}\sum_{\tilde{q}_{\tau+1}\in R_{b,m}}\sum_{\substack{\bsk_{u\cup\{\tau+1\}}\in \nat^{|u|+1}\\ (\bsk_{u\cup\{\tau+1\}},\bszero)\in D^{\perp }}}r_{\alpha,d}^{\lambda}(\bsk_{u\cup\{\tau+1\}},\bszero) \\
	& = \frac{(b-1)^{1+\lambda}}{b^{m+1+\lambda(\alpha-1)}}\sum_{l=m+1}^{\infty}b^{(1-(2\min(\alpha,d)+1)\lambda) l}\sum_{\bsk_u\in \nat^{|u|}}r_{\alpha,d}^{\lambda}(\bsk_u) \nonumber \\
	& \quad + \frac{(b-1)^{1+\lambda}}{(b^m-1)b^{1+\lambda(\alpha-1)}}\sum_{l=1}^{m}b^{(1-(2\min(\alpha,d)+1)\lambda) l}\sum_{\substack{\bsk_u\in \nat^{|u|}\\ \rtr_m(\bsk_u)\cdot \bsq_u \ne 0 \pmod{p}}}r_{\alpha,d}^{\lambda}(\bsk_u) \\
	& \le \frac{(b-1)^{1+\lambda}}{(b^m-1)b^{1+\lambda(\alpha-1)}}\sum_{l=1}^{\infty}b^{(1-(2\min(\alpha,d)+1)\lambda) l}\sum_{\bsk_u\in \nat^{|u|}}r_{\alpha,d}^{\lambda}(\bsk_u) \\
	& = \frac{(b-1)^{1+\lambda}}{(b^m-1)b^{\lambda(\alpha-1)}}\cdot \frac{1}{b^{(2\min(\alpha,d)+1)\lambda}-b}\prod_{j\in u}\left[ \sum_{k_j=1}^{\infty}r_{\alpha,d}^{\lambda}(k_j)\right]. 
	\end{align*}
Here the sum in the product is given by
	\begin{align*}
	\sum_{k_j=1}^{\infty}r_{\alpha,d}^{\lambda}(k_j) & = (b-1)^{\lambda}b^{-\lambda(\alpha-1)}\sum_{l_j=1}^{\infty}b^{-(2\min(\alpha,d)+1)\lambda l_j}\sum_{k_j=b^{l_j-1}}^{b^{l_j}-1}1 \\
	& = \frac{(b-1)^{1+\lambda}}{b^{1+\lambda(\alpha-1)}}\sum_{l_j=1}^{\infty}b^{(1-(2\min(\alpha,d)+1)\lambda) l_j} \\
	& = \frac{(b-1)^{1+\lambda}}{b^{\lambda(\alpha-1)}}\cdot \frac{1}{b^{(2\min(\alpha,d)+1)\lambda}-b}. 
	\end{align*}
Thus, from (\ref{eq:theta}) we obtain 
	\begin{align}
	\theta^{\lambda}(q_{\tau+1}) & \le \frac{1}{b^m-1}\sum_{u\subseteq \{1:\tau\}}\gamma_{v(u\cup\{\tau+1\})}^{\lambda}D_{\alpha,d}^{\lambda|v(u\cup\{\tau+1\})|}\tilde{C}_{\alpha,d,\lambda}^{|u|+1} .
	\label{eq:theta2}
	\end{align}
Recall that $\tau=(j_1-1)d+d_1-1$. Let $J_1:= \{1:(j_1-1)d \}$ and $J_2:= \{(j_1-1)d+1:(j_1-1)d+d_1-1 \}$. In case of $d_1=1$, the set $J_2$ is taken to be the empty set. Every subset $u\subseteq \{1:\tau\}$ can be split into a subset of $J_1$ and a subset of $J_2$. Since $\{\tau+1\}$ is one of $d$ components for the $j_1$-th coordinate, whether or not $u$ includes some element of $J_2$ does not affect $v(u\cup\{\tau+1\})$. From this observation, we have 
	\begin{align*}
	& \quad \sum_{u\subseteq \{1:\tau\}}\gamma_{v(u\cup\{\tau+1\})}^{\lambda}D_{\alpha,d}^{\lambda|v(u\cup\{\tau+1\})|}\tilde{C}_{\alpha,d,\lambda}^{|u|+1} \\
    & = \sum_{u_2\subseteq J_2}D_{\alpha,d}^{\lambda}\tilde{C}_{\alpha,d,\lambda}^{|u_2|+1}\sum_{u_1\subseteq J_1}\gamma_{v(u_1)\cup \{j_1\}}^{\lambda}D_{\alpha,d}^{\lambda|v(u_1)|}\tilde{C}_{\alpha,d,\lambda}^{|u_1|}.
	\end{align*}
By considering the terms associated with a certain $u$ ($u \subseteq \{1:j_1-1\}$) in the inner sum, at least one component of $\{(j-1)d+1:jd\}$ for all $j\in u$ must be chosen. Thus, 
	\begin{align*}
	& \quad \sum_{u\subseteq \{1:\tau\}}\gamma_{v(u\cup\{\tau+1\})}^{\lambda}D_{\alpha,d}^{\lambda|v(u\cup\{\tau+1\})|}\tilde{C}_{\alpha,d,\lambda}^{|u|+1} \\
	& = D_{\alpha,d}^{\lambda}\tilde{C}_{\alpha,d,\lambda}( 1+\tilde{C}_{\alpha,d,\lambda})^{d_1-1} \sum_{u\subseteq \{1:j_1-1\}}\gamma_{u\cup \{j_1\}}^{\lambda}\prod_{j\in u}D_{\alpha,d}^{\lambda}\left( -1+( 1+\tilde{C}_{\alpha,d,\lambda})^d \right) \\
	& = \left(C_{\alpha,d,\lambda,d_1}-C_{\alpha,d,\lambda,d_1-1}\right) \sum_{u\subseteq \{1:j_1-1\}}\gamma_{u\cup \{j_1\}}^{\lambda}C_{\alpha,d,\lambda,d}^{|u|}.
	\end{align*}
Finally, we have
	\begin{align*}
	& \quad B_{\alpha,d, \bsgamma}(\bsq_{\tau+1}, p) = B_{\alpha,d, \bsgamma}(\bsq_{\tau}, p)+\theta(q_{\tau+1}) \\
	& \le \frac{1}{(b^m-1)^{1/\lambda}} \left[ \sum_{\emptyset \ne u\subseteq \{1:j_0-1\}}\gamma_u^\lambda C_{\alpha,d,\lambda,d}^{|u|}+C_{\alpha,d,\lambda,d_0}\sum_{u\subseteq \{1:j_0-1\}}\gamma_{u\cup\{j_0\}}^\lambda C_{\alpha,d,\lambda,d}^{|u|}\right]^{1/\lambda} \\
	& \quad + \frac{1}{(b^m-1)^{1/\lambda}} \left[ \left(C_{\alpha,d,\lambda,d_1}-C_{\alpha,d,\lambda,d_1-1}\right) \sum_{u\subseteq \{1:j_1-1\}}\gamma_{u\cup \{j_1\}}^{\lambda}C_{\alpha,d,\lambda,d}^{|u|}\right]^{1/\lambda} .
	\end{align*}
In case of $0< d_0<d$, we have $j_1=j_0$ and $d_1=d_0+1$. Using Jensen's inequality, we obtain
	\begin{align*}
	B_{\alpha,d, \bsgamma}(\bsq_{\tau+1}, p) \le \frac{1}{(b^m-1)^{1/\lambda}} \left[ \sum_{\emptyset \ne u\subseteq \{1:j_1-1\}}\gamma_u^\lambda C_{\alpha,d,\lambda,d}^{|u|}+C_{\alpha,d,\lambda,d_1}\sum_{u\subseteq \{1:j_1-1\}}\gamma_{u\cup\{j_1\}}^\lambda C_{\alpha,d,\lambda,d}^{|u|}\right]^{\frac{1}{\lambda}} .
	\end{align*}
In case of $d_0=d$, we have $j_1=j_0+1$ and $d_1=1$. Again by using Jensen's inequality, we obtain
	\begin{align*}
	\quad B_{\alpha,d, \bsgamma}(\bsq_{\tau+1}, p) & \le \left[ \sum_{\emptyset \ne u\subseteq \{1:j_1-1\}}\gamma_u^\lambda C_{\alpha,d,\lambda,d}^{|u|}\right]^{\frac{1}{\lambda}} + \left[ C_{\alpha,d,\lambda,d_1}\sum_{u\subseteq \{1:j_1-1\}}\gamma_{u\cup \{j_1\}}^{\lambda}C_{\alpha,d,\lambda,d}^{|u|}\right]^{\frac{1}{\lambda}} \\
	& \le \left[ \sum_{\emptyset \ne u\subseteq \{j_1-1\}}\gamma_u^\lambda C_{\alpha,d,\lambda,d}^{|u|}+C_{\alpha,d,\lambda,d_1}\sum_{u\subseteq \{1:j_1-1\}}\gamma_{u\cup\{j_1\}}^\lambda C_{\alpha,d,\lambda,d}^{|u|}\right]^{\frac{1}{\lambda}} .
	\end{align*}
Hence, the proof is complete.
\end{proof}

\begin{remark}\label{remark:propagation}
We have shown that we can construct an interlaced polynomial lattice rule which satisfies
	\begin{align*}
	B_{\alpha,d, \bsgamma}(\bsq, p) \le  A_{\alpha,d,\bsgamma,\delta}b^{-(2\min(\alpha,d)+1)m+\delta} ,
	\end{align*}
for all $\delta > 0$. Let $\alpha \le \alpha' \le d$ and $\gamma'_w D_{\alpha',d}^{|w|}=(\gamma_w D_{\alpha,d}^{|w|})^{\frac{1+2\alpha'}{1+2\alpha}}$ for all $w\subseteq \{1:s\}$. We simply write $\bsgamma'=(\gamma'_w)_{w\subseteq \{1:s\}}$ and $\bsgamma=(\gamma_w)_{w\subseteq \{1:s\}}$. It follows from Jensen's inequality that
	\begin{align*}
	B_{\alpha',d, \bsgamma'}(\bsq, p) & =  \sum_{\emptyset\ne u\subseteq \{1:ds\}}\gamma'_{v(u)}D_{\alpha',d}^{|v(u)|}\sum_{\substack{\bsk_u\in \nat^{|u|}\\ (\bsk_u,\bszero)\in D^{\perp }}}r_{\alpha',d}(\bsk_u,\bszero) \\
    & \le  \sum_{\emptyset\ne u\subseteq \{1:ds\}}(\gamma_{v(u)} D_{\alpha,d}^{|v(u)|})^{\frac{1+2\alpha'}{1+2\alpha}}\sum_{\substack{\bsk_u\in \nat^{|u|}\\ (\bsk_u,\bszero)\in D^{\perp }}}r_{\alpha,d}^{\frac{1+2\alpha'}{1+2\alpha}}(\bsk_u,\bszero) \\
	& \le B_{\alpha,d,\bsgamma}(\bsq, p)^{\frac{1+2\alpha'}{1+2\alpha}} \\
	& \le A_{\alpha,d,\bsgamma,\delta}^{\frac{1+2\alpha'}{1+2\alpha}}b^{-(2\alpha'+1)m+ \frac{1+2\alpha'}{1+2\alpha}\delta}.
	\end{align*}
for all $\delta > 0$. This means that interlaced polynomial lattice rules constructed component-by-component for functions of smoothness $\alpha$ using an interlacing factor of $d$ still achieve the optimal rate of convergence for functions of smoothness $\alpha'$ as long as $\alpha \le \alpha' \le d$ holds. Our observation is similar to that of the classical polynomial lattice rule shown by \cite{BD11}, while we note that it is opposite from propagation rules \cite[Theorem~3.3]{D07} which states that a higher order net which achieves an optimal rate of convergence for function of smoothness $\alpha$ can achieve an optimal rate of convergence for function of smoothness $\alpha'$ for all $1\le \alpha'\le \alpha$.
\end{remark}

\subsection{Fast construction for product weights}\label{subsec_fast_cbc}

We now show how one can use the fast component-by-component construction to find suitable polynomials $q_1, \ldots, q_{ds} \in \FF_b[x]$ of degree less than $m$ for product weights. From Definition \ref{def:interlacing_polynomial_lattice} and Lemma \ref{lemma:criterion}, we have
\begin{align*}
& \quad B_{\alpha,d, \bsgamma}((q_1,\ldots, q_{ds}),p) \\ 
& = -1 + \frac{1}{b^m} \sum_{n=0}^{b^m-1} \prod_{j=1}^s \Big[1- \gamma_jD_{\alpha,d} + \gamma_j D_{\alpha,d} \prod_{k=1}^d \left( 1 + \phi_{\alpha,d}\left( v_m\left(\frac{n(x)q_{(j-1)d +k}(x)}{p(x)} \right)\right)\right)\Big].
\end{align*}

According to Algorithm \ref{algorithm:cbc}, set $q_1 =1$ and construct the polynomials $q_2, \ldots, q_{ds}$ inductively in the following way. Assume that $q_2, \ldots, q_{\tau}$ are already found. Let $\tau=(j_0-1)d+d_0$ and $\tau+1=(j_1-1)d+d_1$ such that $j_0,d_0,j_1,d_1 \in \nat$ and $0<d_0,d_1\le d$. As in the proof of Theorem \ref{theorem:cbc_proof}, $(j_1,d_1)=(j_0+1,1)$ if $d_0=d$, or otherwise $(j_1,d_1)=(j_0,d_0+1)$. Here we introduce the following notation
\begin{align*}
P_{n,\tau} := \prod_{j=1}^{j_1 - 1} \Big[1-\gamma_j D_{\alpha,d} + \gamma_j D_{\alpha,d}\prod_{k=1}^d \left( 1 + \phi_{\alpha,d}\left( v_m\left(\frac{n(x)q_{(j-1)d +k}(x)}{p(x)} \right)\right)\right)\Big] ,
\end{align*}
and
\begin{align*}
Q_{n,\tau} := \prod_{k=1}^{d_1-1} \left( 1 + \phi_{\alpha,d}\left(v_m\left(\frac{n(x)q_{(j_1-1)d +k}(x)}{p(x)}  \right) \right)\right) ,
\end{align*}
for $0\le n< b^m$. We note that $Q_{n,\tau}=1$ when $d_1=1$ (or $d_0=d$).

Since $p$ is an irreducible polynomial over $\FF_b[x]$, there exists a primitive polynomial $g$ in $\FF_b[x]/p$, that is $\{g^0(x) = g^{b^m-1}(x)=1, g^1(x), \ldots, g^{b^m-2}(x) \} = (\FF_b[x]/p)\setminus\{0\}$. Using the above notation, we have
\begin{align*}
& \quad B_{\alpha,d, \bsgamma}((\bsq_{\tau}, g^z),p) \\
& = -1 + \frac{1}{b^m} \sum_{n=0}^{b^m-1} P_{n,\tau} \Big[ 1- \gamma_{j_1}D_{\alpha,d} + \gamma_{j_1}D_{\alpha,d} Q_{n,\tau} \left( 1 + \phi_{\alpha,d}\left(v_m\left(\frac{g^{z}(x) g^{-n}(x) }{p(x)}  \right) \right) \right)\Big]
\end{align*}
for $1 \le z < b^m$, where $g^{-1}(x) = g^{b^m-2}(x)$ is the multiplicative inverse of $g(x)$ in $\FF_b[x]/p$ (which is also primitive). The aim here is to compute $B_{\alpha,d, \bsgamma}((\bsq_{\tau}, g^z),p)$ for $1 \le z < b^m$ and choose $z_0$ such that $B_{\alpha,d, \bsgamma}((\bsq_{\tau}, g^{z_0}),p) \le B_{\alpha,d, \bsgamma}((\bsq_{\tau}, g^z),p)$ for all $1 \le z < b^m$. Since we only need to compare the values of $B_{\alpha,d, \bsgamma}((\bsq_{\tau}, g^z),p)$ for different values of $z$, we only need to compute the terms which depend on $z$. That is, it is sufficient to compute
\begin{align*}
& \sum_{n=1}^{b^m-1} P_{n,\tau}Q_{n,\tau}\phi_{\alpha,d}\left(v_m\left(\frac{g^{z-n}(x) }{p(x)}  \right)\right).
\end{align*}

We define the circulant matrix
\begin{equation*}
A = \left(\phi_{\alpha,d}\left(v_m\left(\frac{g^{z-n}(x) }{p(x)} \right)\right)\right)_{1 \le z, n < b^m} ,
\end{equation*}
and $\bsa = (a_1,\ldots, a_{b^m-1})^\top$ with
\begin{align*}
a_n = P_{n,\tau}Q_{n,\tau} .
\end{align*}
Let $\bsb = A \bsa$ with $\bsb = (b_1,\ldots, b_{b^m-1})^\top$. Then $z_0$ is the integer $1 \le z_0 < b^m$ which satisfies $b_{z_0} \le b_z$ for $1 \le z < b^m$. Therefore we set $q_{\tau+1} = g^{z_0}$.

Since the matrix $A$ is circulant, the matrix vector multiplication $A \bsa$ can be done using the fast Fourier transform as shown in \cite{NC06a,NC06b}. Thus we obtain a fast computation of the vector $\bsb$.

After finding $q_{\tau+1}$, $P_{n,\tau}$ and $Q_{n,\tau}$ for $0\le n< b^m$ are updated as follows. If $d_1=d$,
\begin{align*}
	\left\{ \begin{array}{ll}
      P_{n,\tau+1}  = & P_{n,\tau}\Big[ 1- \gamma_{j_1}D_{\alpha,d} + \gamma_{j_1}D_{\alpha,d}Q_{n,\tau}\left(1+ \phi_{\alpha,d}\left(v_m\left(\frac{n(x)q_{\tau+1}(x)}{p(x)}  \right) \right)\right)\Big] , \\
      Q_{n,\tau+1}  = & 1 . \\
	\end{array} \right.
\end{align*}
Otherwise if $0<d_1<d$,
\begin{align*}
	\left\{ \begin{array}{ll}
      P_{n,\tau+1} & = P_{n,\tau} , \\
      Q_{n,\tau+1} & = Q_{n,\tau}\left(1+ \phi_{\alpha,d}\left(v_m\left(\frac{n(x)q_{\tau+1}(x)}{p(x)}  \right) \right)\right) . \\
	\end{array} \right.
\end{align*}
Then, we proceed to the next component. Unlike for classical polynomial lattice rules, such as \cite{BD11,NC06b}, here we are required to store  not one but two vectors $(P_{n,\tau})$ and $(Q_{n,\tau})$ in memory. By this slight increase in memory, the fast CBC construction using the fast Fourier transform can be applied. The construction cost is of order $\mathcal{O}(d s m b^m)$ operations using $\mathcal{O}(b^m)$ memory. This compares favorably with the construction of deterministic higher order polynomial lattice point sets in \cite{BDLNP} where the construction cost was of order $\mathcal{O}(\alpha s N^\alpha \log N)$ operations using $\mathcal{O}(N^\alpha)$ memory.

The implementation of interlaced scrambled polynomial lattice rules also requires an efficient implementation of the scrambling procedure. Since our results also hold for the simplifications of the scrambling scheme discussed in \cite{H96,Mat98, O03}, computationally efficient algorithms are available for this purpose. We again refer to \cite{H96,Mat98,O03} for a discussion of the computational efficiency of the various methods.

\section{Numerical experiments}
\label{Experiments}

Finally, we present some numerical results for the bound $B_{\alpha,d,\bsgamma}(\bsq, p)$ on the variance of the estimator in weighted function spaces with smoothness $\alpha\ge 1$. In our computation, the prime base $b$ is always fixed at 2 and only product weights are considered. As a reference, we also compute the following quality criterion $B_{\alpha,d,\bsgamma}(C_1,\ldots,C_{ds})$ by using the first $2^m$ terms of a Sobol' sequence \cite{S67}:
  \begin{align*}
     B_{\alpha,d, \bsgamma}(C_1,\ldots,C_{ds}) = \sum_{\emptyset\ne u\subseteq \{1:ds\}} \gamma_{v(u)}D_{\alpha,d}^{|v(u)|}\sum_{\substack{\bsk_u\in \nat^{|u|}\\ (\bsk_u,\bszero)\in D^{\perp }(C_1,\ldots,C_{ds})}}r_{\alpha,d}(\bsk_u,\bszero) ,
  \end{align*}
where $C_1,\ldots,C_{ds}$ denote the $m\times m$ generating matrices over $\FF_2$ of the $ds$-dimensional Sobol' sequence, and $D^{\perp }(C_1,\ldots,C_{ds})$ denotes the dual net of the first $2^m$ terms of Sobol' sequence.

In Figure \ref{fig:1}--\ref{fig:4}, we show the values of $B_{\alpha,d,\bsgamma}(\bsq,p)$ and $B_{\alpha,d,\bsgamma}(C_1,\ldots,C_{ds})$ from $m=4$ to $m=16$ with various choices of $\alpha$, $d$, $\bsgamma$ and $s$, where the values of $B_{\alpha,d,\bsgamma}(\bsq,p)$ are obtained by using the fast CBC construction. As proved in Theorem \ref{theorem:cbc_proof}, the CBC construction of interlaced scrambled polynomial lattice rules achieves the variance of the estimator of order $N^{-(2\min(\alpha,d)+1)+\delta}$ ($\delta >0$). Since higher order scrambled Sobol' point sets can also achieve the optimal convergence rate of the variance as shown in \cite{D11}, our comparison is reasonable.

Figure \ref{fig:1} shows the results for $s=1$ and $\gamma_1=D_{\alpha,d}^{-1}$ with various choices of $(\alpha, d)$. When $\alpha=1$ (or $d=1$), the decay rate is of order $N^{-3}$ as predicted by the theory. As $\alpha$ and $d$ increase simultaneously, the convergence rate increases to around $N^{-5}$ and $N^{-7}$ for $d=2$ and $d=3$, respectively, which is in accordance with our theory. There is no clear difference between $B_{\alpha,d,\bsgamma}(\bsq,p)$ and $B_{\alpha,d,\bsgamma}(C_1,\ldots,C_{ds})$.

Figure \ref{fig:2} shows the results for $s=2$ and $\gamma_1=\gamma_2= D_{\alpha,d}^{-1}$ with various choices of $(\alpha, d)$. In this case, we can see better convergence behaviors of $B_{\alpha,d,\bsgamma}(\bsq,p)$ as compared to $B_{\alpha,d,\bsgamma}(C_1,\ldots,C_{ds})$, and the almost optimal order of the convergence rate are achieved for our constructed point sets.

In Figures \ref{fig:3} and \ref{fig:4}, we compare the values of $B_{\alpha,d,\bsgamma}(\bsq,p)$ and $B_{\alpha,d,\bsgamma}(C_1,\ldots,C_{ds})$ with $(\alpha,d)=(2,2)$ from $s=1$ to $s=5$ for two different product weights, respectively. One is $\gamma_j=D_{\alpha,d}^{-1}$, the other is $\gamma_j=D_{\alpha,d}^{-1}j^{-2}$. The latter implies a decreasing importance of the successive coordinates. We can see again better convergence behaviors of $B_{\alpha,d,\bsgamma}(\bsq,p)$ as compared to $B_{\alpha,d,\bsgamma}(C_1,\ldots,C_{ds})$. It is clear that a better convergence of $B_{\alpha,d,\bsgamma}(\bsq,p)$ can be observed for decreasing weights (see Figure~\ref{fig:4}). This is reasonable, since our algorithm allows us to adjust our rules to the weights (which is not the case for the Sobol' sequence).

\begin{figure}
\begin{center}
\includegraphics[width=12cm]{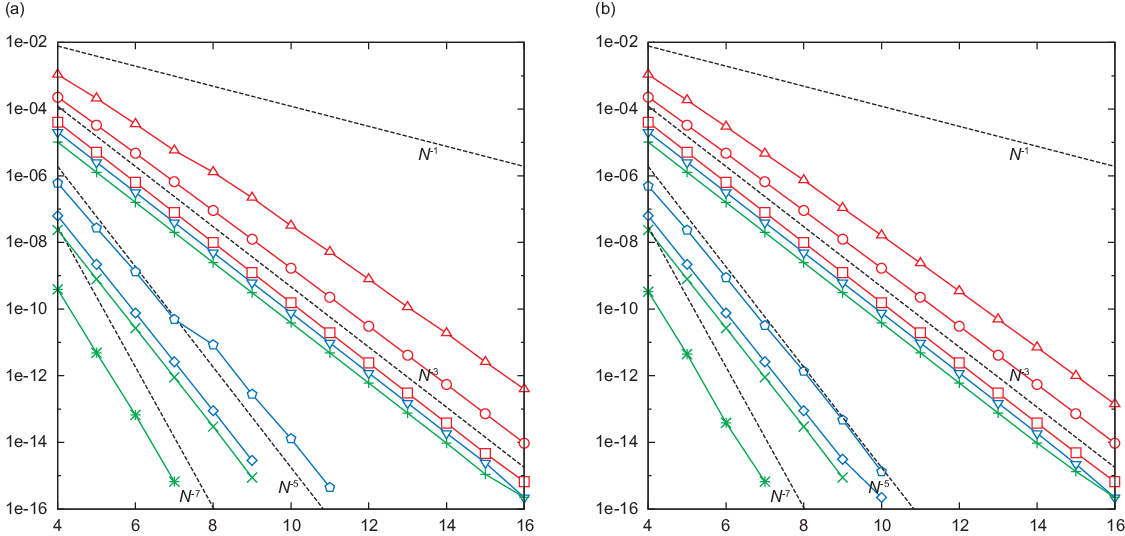}
\caption{Values of $B_{\alpha,d,\bsgamma}(\bsq,p)$ (left) and $B_{\alpha,d,\bsgamma}(C_1,\ldots,C_{ds})$ (right) for $s=1$ and $\gamma_1=D_{\alpha,d}^{-1}$ with various choices of $(\alpha,d)=(1,1), (1,2), (1,3), (2,1), (2,2), (2,3), (3,1), (3,2), (3,3)$, marked respectively by square, circle, triangle, down triangle, diamond, pentagon, plus sign, cross, and asterisk.}
\label{fig:1}
\end{center}
\end{figure}
\begin{figure}
\begin{center}
\includegraphics[width=12cm]{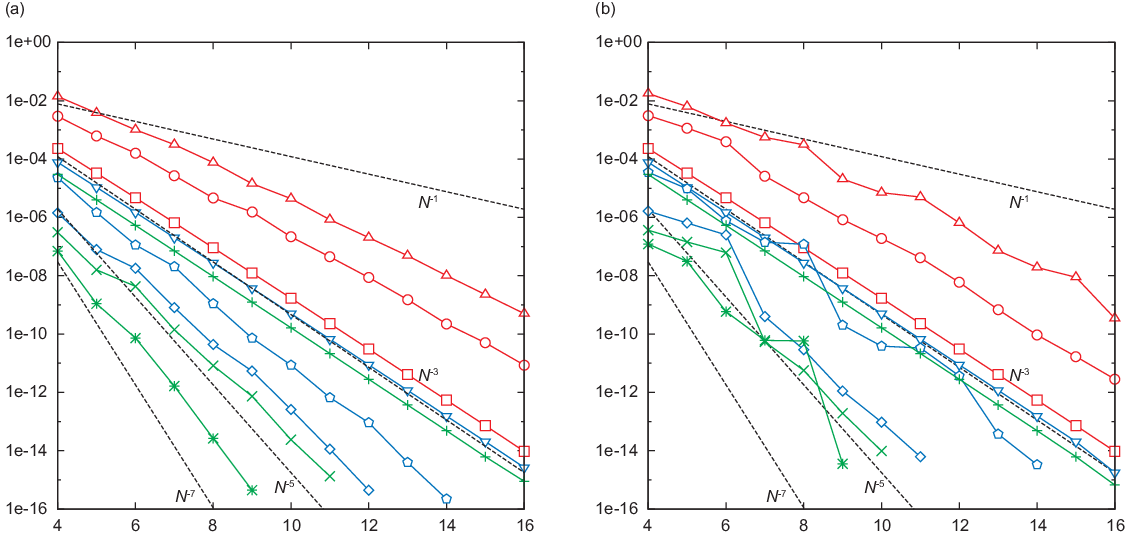}
\caption{Values of $B_{\alpha,d,\bsgamma}(\bsq,p)$ (left) and $B_{\alpha,d,\bsgamma}(C_1,\ldots,C_{ds})$ (right) for $s=2$ and $\gamma_1=\gamma_2=D_{\alpha,d}^{-1}$ with various choices of $(\alpha,d)=(1,1), (1,2), (1,3), (2,1), (2,2), (2,3), (3,1), (3,2), (3,3)$, marked respectively by square, circle, triangle, down triangle, diamond, pentagon, plus sign, cross, and asterisk.}
\label{fig:2}
\end{center}
\end{figure}
\begin{figure}
\begin{center}
\includegraphics[width=12cm]{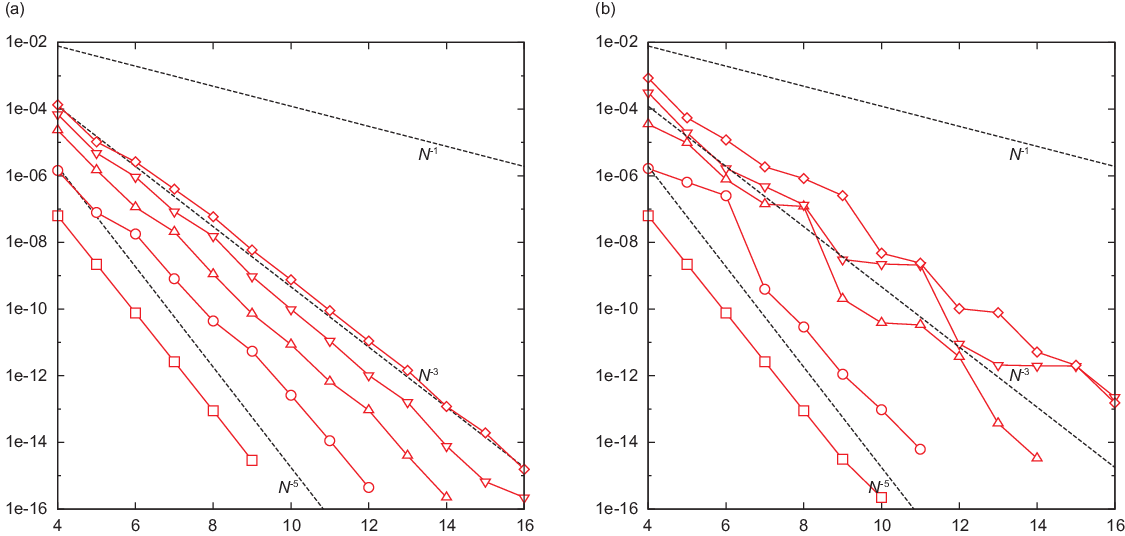}
\caption{Values of $B_{\alpha,d,\bsgamma}(\bsq,p)$ (left) and $B_{\alpha,d,\bsgamma}(C_1,\ldots,C_{ds})$ (right) with $(\alpha,d)=(2,2)$ and $\gamma_j=D_{\alpha,d}^{-1}$ from $s=1$ to $s=5$, marked respectively by square, circle, triangle, down triangle and diamond.}
\label{fig:3}
\end{center}
\end{figure}
\begin{figure}
\begin{center}
\includegraphics[width=12cm]{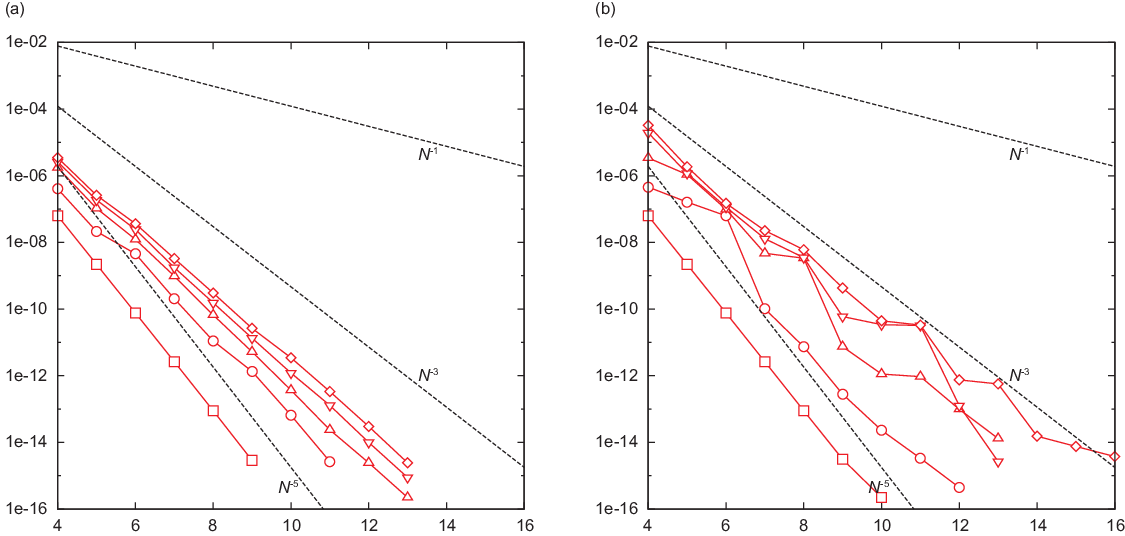}
\caption{Values of $B_{\alpha,d,\bsgamma}(\bsq,p)$ (left) and $B_{\alpha,d,\bsgamma}(C_1,\ldots,C_{ds})$ (right) with $(\alpha,d)=(2,2)$ and $\gamma_j=D_{\alpha,d}^{-1}j^{-2}$ from $s=1$ to $s=5$, marked respectively by square, circle, triangle, down triangle and diamond.}
\label{fig:4}
\end{center}
\end{figure}

\section*{Acknowledgment}

The first author is supported by JSPS Grant-in-Aid for JSPS Fellows No.24-4020 and the second author is supported by a Queen Elizabeth 2 Fellowship from the Australian Research Council. T.G. would like to thank Josef Dick for his hospitality while visiting the University of New South Wales where this research was carried out.


\end{document}